\documentclass[11pt,reqno]{amsart}
\usepackage{amsmath,amsthm,mathtools,amssymb,mathrsfs,color,enumerate,esint}
\usepackage[margin=1.25in]{geometry}

\DeclareMathOperator\Id{Id}
\newcommand\eq[1]{\begin{align}{#1}\end{align}}
\newcommand\eqn[1]{\begin{align*}{#1}\end{align*}}
\newcommand\grad[0]{\nabla}
\let\d\relax
\newcommand\d[0]{\partial}
\let\div\relax
\DeclareMathOperator\div{div}
\DeclareMathOperator\curl{curl}

\DeclareMathOperator\dist{dist}
\DeclareMathOperator\im{Im}

\newcommand\Td[0]{{\mathbb T^3}}
\newcommand\Rd[0]{{\mathbb R^3}}
\newcommand\Zd[0]{{\mathbb Z^3}}

\theoremstyle{plain}
\newtheorem{thm}{Theorem}[section]
\newtheorem{lem}[thm]{Lemma}
\newtheorem{cor}[thm]{Corollary}
\newtheorem{proposition}[thm]{Proposition}

\theoremstyle{definition}
\newtheorem{define}[thm]{Definition}

\theoremstyle{remark}
\newtheorem{remark}[thm]{Remark}

\makeatletter
\renewcommand*\env@matrix[1][\arraystretch]{%
  \edef\arraystretch{#1}%
  \hskip -\arraycolsep
  \let\@ifnextchar\new@ifnextchar
  \array{*\c@MaxMatrixCols c}}
\makeatother

\usepackage{tikz}
\usetikzlibrary{arrows.meta, positioning, calc}
\usepackage{subcaption}
\usepackage{adjustbox}
\definecolor{bluee}{HTML}{4682B4}

\title{Arbitrary norm growth in the 3D Navier--Stokes equations}

\author{Stan Palasek}
\address{Department of Mathematics, Princeton University \& Institute for Advanced Study, 1 Einstein Dr., Princeton, NJ 08540}
\email{palasek@ias.edu}

\begin{document}

\begin{abstract}
    We construct a family of smooth initial data for the Navier--Stokes equations, bounded in $BMO^{-1}(\mathbb T^3)$, that gives rise to arbitrarily large global solutions. As a consequence, we rule out various hypothetical a~priori estimates for strong solutions in terms of a critical norm of the initial data. To our knowledge, this is the first example of unbounded norm growth in the well-posed setting. The solutions exhibit an inverse cascade across an unbounded number of modes, with growth resulting from repeated squaring by the quadratic nonlinearity. This mechanism relies on largeness of the data in $B^{-1}_{\infty,\infty}$ and is fundamentally distinct from instantaneous norm inflation and ill-posedness phenomena. 
\end{abstract}

\maketitle

\section{Introduction}

Consider the incompressible Navier--Stokes equations
\begin{equation}\begin{aligned}\label{nse}
\d_t u-\nu\Delta u+u\cdot\grad u+\grad p=0\\
\div u=0\\
u(x,0)=u^0(x)
\end{aligned}\end{equation}
on the three-dimensional torus $\Td=(\mathbb R/2\pi\mathbb Z)^3$, equipped with divergence-free initial data $u^0\in C^\infty(\Td;\Rd)$. By a suitable rescaling of $t$, $u$, $u^0$, and $p$, we may normalize the viscosity to $\nu=1$. To avoid technicalities, we fix the discussion on the Cauchy problem for classical solutions, i.e., smooth vector fields $u(x,t)$ on $\Td\times[0,T)$ for some $T>0$. It has been known since the seminal work of Leray \cite{leray1934mouvement} that for any such initial data, there exists a unique solution for at least a time $T>0$ depending on, for instance, $\|u^0\|_{L^\infty(\Td)}$. Furthermore, there is a maximal $T_*\in(0,\infty]$ such that if $T_*<\infty$, then $u$ cannot be continued smoothly to any $T_*+\epsilon$. A fundamental question in the theory of the Navier--Stokes equations is whether $T_*=\infty$ in all cases.

The problem of global well-posedness as posed above is soft in the sense that it is not overtly connected to any quantitative properties of $u$. It is conceivable, for instance, that there is a family of data bounded in a space $X$ that becomes arbitrarily large by some measure, but blow-up is always averted by some qualitative mechanism. This scenario was ruled out by Tao \cite{tao-quantitative-formulation-navier} for the case when both the initial data and subsequent growth are measured in $H^1(\Td)$. More precisely, he shows that global regularity of \eqref{nse} is equivalent to the existence of a non-decreasing $f:[0,\infty)\to[0,\infty)$ such that
\eq{\label{quantitative-formulation-1}
\|u(T)\|_{X}\leq f(\|u^0\|_X)\tag{E1}
}
holds for all $T\in(0,1)$ and all classical solutions on $\Td\times[0,T]$ with data in $X=H^1$. In \cite{tao-quantitative-formulation-navier} he conjectures that the same is true when the subcritical norm $H^1$ is replaced by a critical norm, i.e., a norm which, when regarded on $\Rd$, is invariant under the scaling symmetry of \eqref{nse}:
\eqn{
u(x,t)\mapsto \lambda u(\lambda x,\lambda^2t),\qquad p(x,t)\mapsto\lambda^2 p(\lambda x,\lambda^2t)
}
for any $\lambda>0$.
  
Alternatively, one can hypothesize a~priori bounds of parabolic smoothing type, namely
\eq{\label{quantitative-formulation-2}
\sup_{t\in(0,T)}t^{\frac12(1+n)}\|\grad^nu(t)\|_{L^\infty(\Td)}\leq f_n(\|u^0\|_{X})\tag{E2}
}
for $n=0,1,2\ldots$. The equivalence of \eqref{quantitative-formulation-2} (for some particular $T$, $n$, and $f_n$, say) with the global regularity of \eqref{nse} was proved for the critical spaces $X=\dot H^{\frac12}(\Rd)$ by Rusin and \v Sver\'ak \cite{rusin-sverak} and $X=L^3(\Rd)$ by Jia and \v Sver\'ak \cite{jia-sverak-minimal}.\footnote{We remark that the results in \cite{rusin-sverak} and \cite{jia-sverak-minimal} are stated in $\mathbb R^3$ and the proofs make some use of the scaling symmetry; nonetheless, we expect they can be straightforwardly transferred to the periodic setting. We likewise expect that the main results in the present work can be extended to decaying solutions in $\Rd$ with some technical modifications.} It is almost certainly the case that the techniques in \cite{rusin-sverak} and \cite{jia-sverak-minimal} can be used to prove the critical analogue of Tao's result for \eqref{quantitative-formulation-1} with $X=\dot H^{\frac12}(\Td)$ and $L^3(\Td)$.

These results strongly suggest that global well-posedness is broadly equivalent to quantitative a~priori control at sufficiently high regularity (e.g., enough to support local well-posedness). A natural question is whether such control (in the form of \eqref{quantitative-formulation-1}, \eqref{quantitative-formulation-2}, or otherwise) persists in borderline critical spaces, and whether its equivalence to global regularity survives in that regime.\footnote{The arguments in \cite{rusin-sverak} and \cite{jia-sverak-minimal} would break down, for instance, in $L^{3,\infty}$ in the event that local well-posedness fails for large data. Indeed, the proof is by contradiction, considering a bounded family of data that exhibits growth and passing to a weak limit along a subsequence. The weak limit needs to have a local solution to contradict extendability to a global solution.} To date there has been little progress in this direction; failure of estimates like \eqref{quantitative-formulation-1} has been shown only in very weak settings where the Navier--Stokes equations are known to be severely ill-posed (e.g., $u^0\in B^{-1}_{\infty,\infty}$; see Section~\ref{previous-work-section} for a survey).

The primary purpose of this work is to show that \eqref{quantitative-formulation-1} and \eqref{quantitative-formulation-2} can dramatically fail even for choices of $X$ where there is local well-posedness and small data global well-posedness, and even from smooth initial data. We focus our discussion on\footnote{The data we construct lie in the stronger space $B^{-1}_{\infty,1}$. However, it is known that this norm is poorly adapted to the Navier--Stokes (see \cite{wang2015ill}) so we do not emphasize this point.} the well-known Koch--Tataru spaces $BMO^{-1}$ and $VMO^{-1}$, originally considered in \cite{koch2001well}. We find that all reasonable a~priori estimates in terms of these norms of the data are false.

The statements on failure of a~priori estimates will be straightforward consequences of the following theorem:

\begin{thm}\label{second-theorem}
Let $\Theta_*\in\Rd$ and $\eta_*\in\Zd$ be any non-zero vectors satisfying $\Theta_*\cdot\eta_*=0$. For any $\epsilon_*>0$ and $n_{max}\in\mathbb N$, there exists divergence-free initial data $u^0\in C^\infty(\Td)$ with
    \eqn{
    \|u^0\|_{B^{-1}_{\infty,1}}\leq 10^5
    }
    and such that the corresponding strong solution of \eqref{nse} is global and satisfies
    \eqn{
    u(x,t)=\Theta_*\sin(x\cdot\eta_*)\exp(-|\eta_*|^2t)+E
    }
    where
    \eqn{
    \|\grad^nE(t)\|_{L^\infty(\Td)}\leq \epsilon_*\qquad\forall t\in\big[\frac12|\eta_*|^{-2},2|\eta_*|^{-2}\big]
    }
    for $n=0,1,2,\ldots,n_{max}$.
\end{thm}

Let us emphasize that the strength of the theorem comes from taking $|\Theta_*|$ arbitrarily large, as this is the size the solution grows to when $t\sim |\eta_*|^{-2}$. Meanwhile, $\epsilon_*$, the size of the error, can independently be made arbitrarily small.

\begin{remark}
    We state the uniform bound on $u^0$ in terms of the critical $B_{\infty,1}^{-1}$ norm; this is stronger than boundedness in the well-known Koch--Tataru class $BMO^{-1}$, recalling the inclusions \eqref{besov-embedding} and \eqref{besov-bmo-inclusion}. Further, recall that $BMO^{-1}\cap C^\infty\subset VMO^{-1}$, a space from which solutions enjoy local well-posedness as shown in \cite{koch2001well}.
    
    We also point out that the upper bound $10^5$ is far from optimal, as we did not make significant effort to optimize the constants in the construction.
\end{remark}

\begin{remark}
    Theorem~\ref{second-theorem} and Corollary~\ref{a-priori-estimate-corollary} likely continue to hold for data that is both bounded in $B^{-1}_{\infty,1}$ \emph{and} small in $B^{-1}_{\infty,\infty}$ by combining the construction below with ideas from \cite{bourgain2008ill}. We leave this extension to future work.
\end{remark}

\begin{remark}
    The conclusion of Theorem~\ref{second-theorem} can be interpreted as the statement that data can be chosen from a ball such that on a particular time scale, the solution $u$ approximates a prescribed global solution of \eqref{nse}. This can almost certainly be strengthened to prescribe that $u$ scatters to \emph{any} global mild solution, but we do not pursue this generalization.
\end{remark}

From Theorem~\ref{second-theorem}, we can easily infer the failure of a variety of natural a~priori estimates, including those of the form \eqref{quantitative-formulation-1} and \eqref{quantitative-formulation-2}. In order to state the strongest possible result, we define the following restrictive class of initial data:
\eqn{
\mathscr{X}_\infty=\big\{u^0\in C_{df}^\infty(\Td) \text{ giving rise to a global strong solution $u(x,t)$ of \eqref{nse}}\big\},
}
where $C_{df}^\infty$ denotes the smooth divergence-free vector fields on $\Td$.

\begin{cor}\label{a-priori-estimate-corollary}
    Consider any $T>0$ and non-decreasing function $f:[0,\infty)\to[0,\infty)$. The following hypothetical a~priori estimates all fail to hold in general.
    \begin{enumerate}[\hspace{.35in}1.]
        \item Estimates of the form \eqref{quantitative-formulation-1}:
        \eq{\label{theorem-apriori-bound-e1}
        \sup_{t\in(0,T)}\|u(t)\|_{BMO^{-1}}\leq f(\|u^0\|_{BMO^{-1}})\qquad\forall u^0\in\mathscr{X}_\infty.
        }
        \item Estimates of the form \eqref{quantitative-formulation-2}:
        \eq{\label{theorem-a-priori-bound-e2}
        \sup_{t\in(0,T)}t^{\frac12(1+n)}\|\grad^nu(t)\|_{\infty}\leq f(\|u^0\|_{BMO^{-1}})\qquad\forall u^0\in\mathscr{X}_\infty,}
        for any integer $n\geq0$.
        \item Estimates on localized enstrophy:
        \eq{\label{theorem-a-priori-bound-enstrophy}
        \sup_{t\in(0,T)}t^{\frac14}\|\omega(t)\|_{L^2(B(0,t^{1/2}))}& \leq f(\|u^0\|_{BMO^{-1}})\qquad\forall u^0\in\mathscr{X}_\infty
        }
        with $\omega=\curl u$ the vorticity.
        \item Estimates on (dynamically localized) Prodi--Serrin--Ladyzhenskaya norms:
        \eq{\label{theorem-apriori-bound-prodiserrin}
        \|u\|_{L^p([T/2,T];L^q(B(t^{1/2})))}\leq f(\|u^0\|_{BMO^{-1}})\qquad\forall u^0\in\mathscr{X}_\infty,
        }
        for any choice of $(p,q)\in[2,\infty]\times[3,\infty]$ with $\frac2p+\frac3q=1$.
    \end{enumerate}
    More strongly, the a~priori bounds \eqref{theorem-apriori-bound-e1}--\eqref{theorem-apriori-bound-prodiserrin} fail even if the fixed time $T$ is replaced by a non-increasing function of the size of the data: $T=T(\|u^0\|_{BMO^{-1}})$. 
\end{cor}

\begin{remark}
    While we state the theorem with the more familiar $BMO^{-1}$ norm on the right-hand side, the results continue to hold when it is replaced with the stronger norm $B^{-1}_{\infty,1}$. In fact, \eqref{theorem-apriori-bound-e1} can be weakened even further to rule out hypothetical bounds of the form
            \eqn{
        \sup_{t\in(0,T)}\|u(t)\|_{B^{-1}_{\infty,\infty}}\leq f(\|u^0\|_{B^{-1}_{\infty,1}})\qquad\forall u^0\in\mathscr{X}_\infty.
        }
        That is, even when measured in the weakest (translation-invariant) critical space, solutions are not a~priori bounded from $BMO^{-1}$ or $B^{-1}_{\infty,1}$ data.
\end{remark}

\begin{remark}
        The estimates \eqref{theorem-apriori-bound-e1}--\eqref{theorem-apriori-bound-prodiserrin} are meant to give a wide section of hypothetical quantities that might be bounded but are certainly not exhaustive; the same result holds with essentially any scale-invariant measure of the solution on the left-hand side.

        Let us comment in particular on \eqref{theorem-apriori-bound-prodiserrin}. The Prodi--Serrin--Ladyzhenskaya norm on the left-hand side (ignoring the spatial localization) plays a similar role to Strichartz norms in the context of critical dispersive equations. Some a~priori bounds analogous to \eqref{theorem-apriori-bound-prodiserrin} for critical wave and Schr\"odinger equations have been established; see \cite{bahouri1999high} and \cite{tao-pseudoconformal}, which were also mentioned in \cite{tao-quantitative-formulation-navier}.
\end{remark}

A final application of Theorem~\ref{second-theorem} pertains to the mild solution theory of Koch and Tataru in which they define the path space
\eq{\label{XT-koch-tataru-definition}
\|u\|_{X_T}\coloneqq \sup_{t\in(0,T]}t^\frac12\|u(t)\|_\infty+\sup_{x_0\in\Td}\sup_{R\in(0,T^\frac12]}\left(R^{-3}\int_0^{R^2}\int_{B(x_0,R)}|u|^2dxdt\right)^\frac12,
}
adapted to initial data $u^0\in BMO^{-1}$. In the celebrated paper \cite{koch2001well}, they show global well-posedness when $\|u^0\|_{BMO^{-1}}$ is small, as well as local well-posedness when it is large and $u^0$ belongs to the more restrictive space $VMO^{-1}$ (the closure of the test functions in the $BMO^{-1}$ norm). The proof is by the Picard method with the mild solution operator
\eqn{
\Phi(u)\coloneqq e^{t\Delta}u^0-\int_0^te^{(t-t')\Delta}\mathbb P\div u\otimes u(t')dt',
}
which they show is a contraction on the ball $B_{X_T}(0,\epsilon)$ for $\epsilon$ sufficiently small and all $T>0$.

In \cite{coic-palasek}, the author with M.\ Coiculescu showed that the Navier--Stokes equations are ill-posed (in the sense of non-uniqueness) in $X_T$ for large data in $BMO^{-1}$. In particular, there exists data $u^0$ such that $\Phi$ fails to be contractive on any bounded subset of $X_T$, for any choice of $T$. Theorem~\ref{second-theorem} allows the following strengthening of this result.

\begin{cor}\label{kt-vmo-bmo-corollary}
    For any $T,R>0$, there exists data $u^0\in B_{VMO^{-1}}(0,r_0)\cap C^\infty$ such that the Picard method with the map $\Phi$ fails to converge in $B_{X_T}(0,R)$, where $r_0$ is some absolute constant.
\end{cor}

In other words, failure of the Picard method for large $BMO^{-1}$ data as shown in \cite{coic-palasek} cannot be avoided by taking $u^0$ in $VMO^{-1}$ or even $C^\infty$. On the other hand, we emphasize that Corollary~\ref{kt-vmo-bmo-corollary} is only a constraint on \emph{this particular method} for constructing solutions from large data; it does not rule out their existence.

\subsection{Norm growth and ill-posedness in the Navier--Stokes equations}\label{previous-work-section}

In this short survey, we focus on the case of critical initial data. Critical spaces lie at the threshold between well- and ill-posedness and are therefore a natural setting to study the Cauchy problem. Furthermore, in a critical space, the norm is dimensionless so it is meaningful to talk about small data. Popular choices of the space\footnote{Strictly speaking, we should specify the homogeneous version of these norms; in practice, however, we will be concerned with zero-average functions on $\Td$ for which there is no distinction. We refer the reader to \S\ref{definitions-notation-section} for definitions of the spaces being used.} include, but are not limited to,
\eqn{
H^\frac12\subset L^3\subset L^{3,\infty}\subset B^{-1+\frac3p}_{p,\infty}\subset BMO^{-1}\subset B^{-1}_{\infty,\infty}
}
where $p\in(3,\infty)$, with continuous embeddings. The regularity theory in this scale is delicate and depends strongly on the exact critical space in question. Roughly speaking, there are three categories:
\begin{description}
    \item[Well-posed spaces] In the most regular spaces\footnote{These spaces are characterized by the property that the number of frequency scales with critically-sized elements is finite (and, in the case of $H^\frac12$ and $L^3$, controlled by the norm). In particular they do not include $-1$-homogeneous functions.}, for instance $H^\frac12$, $L^3$, and $VMO^{-1}$, we have local well-posedness for large data and global well-posedness for sufficiently small data~\cite{fujita1964navier,kato1984strong,koch2001well}.
    \item[Borderline spaces] In the borderline setting, for instance $L^{3,\infty}$, $\dot B^{-1+\frac3p}_{p,\infty}$ ($3<p<\infty$), and $BMO^{-1}$, there is global well-posedness for small data only~\cite{cannone1994ondelettes,planchon1998asymptotic,koch2001well}; for large data, local well-posedness was recently shown to fail in $BMO^{-1}$ \cite{coic-palasek} and remains open in the others.
    \item[Ill-posed spaces] In, for instance, $B^{-1}_{\infty,\infty}$, there is ill-posedness even from small data in the form of norm inflation, whereby arbitrarily small data can grow to any prescribed size~\cite{bourgain2008ill}. This can be considered as a strong counterexample to \eqref{quantitative-formulation-1}.
\end{description}

To this point, norm growth has been shown only of the ``norm inflation'' type, meaning it is confined to the third category of spaces. More precisely, one exhibits, for every $\epsilon>0$, initial data with $\|u^0\|_X<\epsilon$ that grows to $\|u(t)\|_X>\epsilon^{-1}$ at a later time. This was shown in the weakest critical space $X=B^{-1}_{\infty,\infty}$ by Bourgain--Pavlovi\'c in \cite{bourgain2008ill}. Their example exhibits what can be considered ill-posedness in the sense of Hadamard because the data-to-solution map is not continuous at $0$ in $B^{-1}_{\infty,\infty}$. We also mention the results of Germain~\cite{germain2008second} and Yoneda \cite{yoneda2010ill} on the extension to the stronger spaces $B^{-1}_{\infty,q}$ for $q>2$. Later, Wang \cite{wang2015ill} showed norm inflation in the sense that small data in $B^{-1}_{\infty,1}$ can become large in the same space (even while staying small in $BMO^{-1}$ and $B^{-1}_{\infty,\infty}$, as required by the Koch--Tataru theorem; see the embeddings \eqref{besov-embedding} and \eqref{besov-bmo-inclusion}).

In the setting of supercritical data, sharp ill-posedness results in the full scale of Besov spaces were obtained by Luo \cite{luo2024illposedness,luo2025sharp} using a mixing/unmixing construction.

We also mention an interesting observation of Cheskidov and Dai~\cite{cheskidov2014norm} that for hyperdissipative Navier--Stokes (i.e., $(-\Delta)^\gamma$ with $\gamma>1$), the natural space for norm inflation \`a la Bourgain--Pavlovi\'c is $B^{-\alpha}_{\infty,\infty}$ which is subcritical. It seems to be a remarkable coincidence that the scaling demanded by the construction in the present work, as well as those in \cite{bourgain2008ill,cheskidov2014norm,palasek2024non,coic-palasek}, etc., is critical precisely for the standard Navier--Stokes equations ($\gamma=1$).

\subsection{Idea of the construction}\label{ideas-subsection}

Let us emphasize that the previous norm inflation results mentioned in \S\ref{previous-work-section} are only possible in spaces where the Navier--Stokes equations are ill-posed for small data. An analogous result in $BMO^{-1}$, say, where the initial data begins arbitrarily small and becomes arbitrarily large, is ruled out by \cite{koch2001well}. Thus, the construction leading to Theorem~\ref{second-theorem} must rely in an essential way on the data being above some threshold size in $BMO^{-1}$.

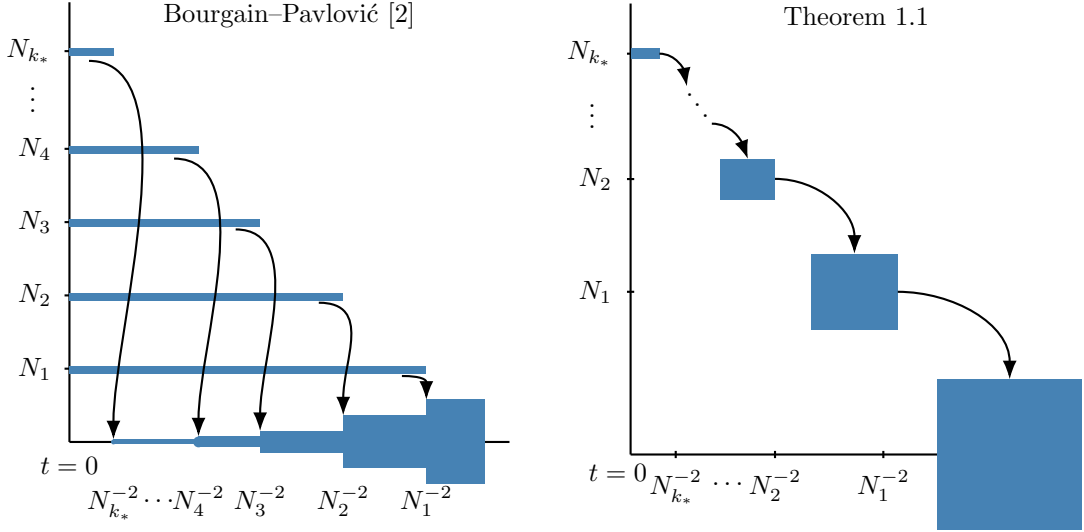
\begin{figure}[htbp]
    \centering 
    
    \begin{subfigure}{0.49\textwidth}
        \centering
\begin{tikzpicture}[
    >=Latex, 
    font=\small, 
    scale=.65
]
    \node[anchor=north] at (4.5, 9.2) {Bourgain--Pavlovi\'c~\cite{bourgain2008ill}};

    \draw[-, thick] (0,0) -- (0,9); 
    \draw[-, thick] (0,0) -- (9,0) node[anchor=north east] {}; 

    \node[below=2pt of {(0,0)}] {$t=0$};
    \node at (1.87, -1.20) {$\dots$}; 

    \node at (-.75, 7.17) {\vdots};

    \foreach \i in {1, 2, 3} {
        \pgfmathsetmacro{\ypos}{1.5 * \i}
        \pgfmathsetmacro{\xpos}{9 - 1.7*\i}
        \pgfmathsetmacro{\radius}{-.1 + 0.32*(4-\i)}

        \draw[thick] (-0.1, \ypos) -- (0.1, \ypos);
        \node[left=3pt of {(0, \ypos)}] {$N_{\i}$};
        \fill[bluee] (0, \ypos-0.1) rectangle (\xpos, \ypos+0.05);
        \node (circle\i) at (\xpos, 0) {};

        \fill[bluee] (\xpos, -\radius) rectangle (8.5, \radius);
        \node[below=9pt of circle\i] {$N_{\i}^{-2}$};
        \draw[-{Latex[length=2.5mm]}, thick]
            (\xpos - 0.5, \ypos-0.15) to[out=0, in=90, looseness=1.8-.3*\i] (\xpos, \radius);
    }

    \begin{scope}
        \pgfmathsetmacro{\i}{4}
        \pgfmathsetmacro{\ypos}{1.5 * \i}
        \pgfmathsetmacro{\xpos}{9 - 1.59*\i}
        \pgfmathsetmacro{\radius}{0.11 + 0.1*(4-\i)}

        \draw[thick] (-0.1, \ypos) -- (0.1, \ypos);
        \node[left=3pt of {(0, \ypos)}] {$N_{\i}$};
        \fill[bluee] (0, \ypos-0.1) rectangle (\xpos, \ypos+0.05);
        \node (circle4) at (\xpos, 0) {};
        \fill[bluee] (circle4) circle (\radius);
        \fill[bluee] (\xpos, -\radius) rectangle (8.5, \radius);
        \node[below=9pt of circle4] {$N_{\i}^{-2}$};
        \draw[-{Latex[length=2.5mm]}, thick]
            (\xpos - 0.5, \ypos-0.2) to[out=0, in=90, looseness=.8] (\xpos, \radius);
    \end{scope}

    \begin{scope}
        \pgfmathsetmacro{\yposk}{8.0}
        \pgfmathsetmacro{\xposk}{.9}
        \pgfmathsetmacro{\radiusk}{0.05}

        \draw[thick] (-0.1, \yposk) -- (0.1, \yposk);
        \node[left=3pt of {(0, \yposk)}] {$N_{k_*}$};
        \fill[bluee] (0, \yposk-0.1) rectangle (\xposk, \yposk+0.05);
        \node (circlek) at (\xposk, 0) {};
        \fill[bluee] (circlek) circle (\radiusk);
        \fill[bluee] (\xposk, -\radiusk) rectangle (8.5, \radiusk);
        \node[below=9pt of circlek] {$N_{k_*}^{-2}$};
        \draw[-{Latex[length=2.5mm]}, thick]
            (\xposk - 0.5, \yposk-0.2) to[out=0, in=90, looseness=.7] (\xposk, \radiusk);
    \end{scope}

\end{tikzpicture}
    \end{subfigure}
    \hfill
        \begin{subfigure}{0.49\textwidth}
\begin{tikzpicture}[
    >=Latex, 
    font=\small, 
    thick, 
    scale=.48
]
    \node[anchor=south] at (6.25, 11.6) {Theorem~\ref{second-theorem}};

    \draw[-] (0,0) -- (0,12.5); 
    \draw[-] (0,0) -- (12.5,0); 

    \node at (-.35,-.5) {$t=0$};
    \draw (1.25, -0.1) -- (1.25, 0.1);
    \node[below=2pt of {(1.25,0)}] {$N_{k_*}^{-2}$};
    \draw (4.0, -0.1) -- (4.0, 0.1);
    \node[below=2pt of {(4.0,0)}] {$N_{2}^{-2}$};
    \draw (7.0, -0.1) -- (7.0, 0.1);
    \node[below=2pt of {(7.0,0)}] {$N_{1}^{-2}$};
    \node at (2.75, -0.82) {$\dots$};

    \node[left=3pt of {(0, 4.51)}] {$N_{1}$};   
    \node[left=3pt of {(0, 7.635)}] {$N_{2}$}; 
    \node[left=3pt of {(0, 11.11)}] {$N_{k_*}$};
    \draw (-0.1, 4.51) -- (0.1, 4.51);
    \draw (-0.1, 7.635) -- (0.1, 7.635);
    \draw (-0.1, 11.11) -- (0.1, 11.11);
    \node at (-1.06, 9.6) {\vdots};  

    \fill[bluee] (0, 10.97) rectangle (.8, 11.25); 
    \coordinate (Nk_end) at (.8, 11.11);
    \fill[bluee] (2.5, 7.08) rectangle (4.0, 8.19);
    \coordinate (N2_start) at (3.25, 8.19);
    \coordinate (N2_end) at (4.0, 7.64);

    \fill[bluee] (5.0, 3.47) rectangle (7.4, 5.56);
    \coordinate (N1_start) at (6.2, 5.56); 
    \coordinate (N1_end) at (7.4, 4.51); 

    \fill[bluee] (8.5, -2.08) rectangle (12.5, 2.08); 
    \coordinate (Final_start) at (10.5, 2.08);    

    \begin{scope}
        \coordinate (A1_end) at (1.56, 10.20);
        \coordinate (A2_start) at (2.25, 9.17); 
        \draw[-{Latex[length=2.5mm]}] (Nk_end) to[out=0, in=100,looseness=1.] (A1_end);
        \node[rotate=-15] at (1.91, 9.90) {$\ddots$}; 
        \draw[-{Latex[length=2.5mm]}] (A2_start) to[out=0, in=110, looseness=1.] (N2_start);
    \end{scope}

    \draw[-{Latex[length=2.5mm]}] (N2_end) to[out=0, in=90, looseness=1] (N1_start);

    \draw[-{Latex[length=2.5mm]}] (N1_end) to[out=0, in=90, looseness=1] (Final_start);

\end{tikzpicture}
    \end{subfigure}
    
    \caption{We illustrate the conceptual distinction between the ill-posedness construction in \cite{bourgain2008ill} (left) and the construction leading to Theorem~\ref{second-theorem} (right). The axes represent time (horizontal) and frequency scale (vertical), while the thickness of the rectangles indicates the $B^{-1}_{\infty,\infty}$ norm of $u(t)$ projected to the shell. In both cases, a large concentration accumulates at the lowest frequency: in \cite{bourgain2008ill}, by many single-step cascades; in the present work, via a single iterated cascade.}
    \label{figure}
\end{figure}

Before introducing the ideas toward Theorem~\ref{second-theorem}, we recall the $B^{-1}_{\infty,\infty}$ construction of Bourgain--Pavlovi\'c for the purpose of comparison. (See the left panel in Figure~\ref{figure}.) The main idea is to consider initial data $u^0(x)=\sum_{k=1}^{k_*}\epsilon N_ku^0_k(x)$, where $u^0_k(x)$ are shear flows with bounded, oscillatory profiles supported at a lacunary sequence of frequencies $N_k$, and $k_*$ is a large integer. Clearly such data is $O(\epsilon)$ in $B^{-1}_{\infty,\infty}$. Evolving time forward for each $k$ to obtain an (approximate) solution $u_k(x,t)$, it is straightforward to arrange that the only significant nonlinear interactions are the self-interactions of the $u_k(x,t)$, and that these interactions strictly send energy to some fixed low frequency shell, say $N_0$. Each $k$ can increase\footnote{Here we are making heuristic use of the Littlewood--Paley projections; see \S\ref{definitions-notation-section} for the relevant notation.} $\|P_{\sim N_0}u\|_{L^\infty}$ by $\sim \epsilon^2N_0$. Taking $k_*\gg\epsilon^{-3}$, one attains the desired norm growth.

In \cite{bourgain2008ill}, growth in the $N_0$ frequency shell is driven by the high-high-to-low nonlinear interactions $P_{\sim N_0}\mathbb P\div (P_{\sim N_k}u)\otimes (P_{\sim N_k}u)$ across all $k\in\{1,2,\ldots,k_*\}$, with the largeness coming from the number of contributing scales $k_*$. Unfortunately, taking $k_*$ large also makes the data large in $BMO^{-1}$ and $B^{-1}_{\infty,q}$ for $q<\infty$. This tradeoff becomes critical at $q=2$ (cf., the embedding \eqref{besov-bmo-inclusion}).

In contrast, our construction leading to Theorem~\ref{second-theorem} uses data supported \emph{only} in the frequency shell $N_{k_*}$. (See the right panel in Figure~\ref{figure}.) The data is constructed so that it generates a particular perturbation at the smaller frequency shell $N_{k_*-1}$ via the high-high-low interaction $P_{N_{k_*-1}}\mathbb P\div (P_{N_{k_*}}u)\otimes (P_{N_{k_*}}u)$. Once this perturbation has appeared, but before it is dissipated by the Laplacian, it will generate yet another perturbation at scale $N_{k_*-2}$, and so on down to a minimum scale $N_0$. The perturbation that arrives at the smallest frequency should match the leading term claimed in Theorem~\ref{second-theorem}. We can arrange the time scales in such a way that once the component at the $k$th shell has successfully perturbed the $(k-1)$st, what remains will dissipate away by the heat equation. Note that since there is no force or convex integration-type correction for $t>0$, all of this behavior must be autonomous and encoded in the initial data.

The source of the growth is the quadratic nonlinearity, which effectively squares the $B^{-1}_{\infty,\infty}$ sizes of modes during the inverse cascade. More quantitatively, suppose the initial data $u^0$ (which, recall, is localized in the $N_{k_*}$ frequency shell) has amplitude $\sim C_{k_*}$. If the projection of $u$ to the $N_{k_*}$ shell evolves under the heat flow to leading order, then the force it exerts on the $N_{k_*-1}$ shell is
\eqn{
|P_{N_{k_*-1}}\mathbb P\div (P_{N_{k_*}}u)\otimes (P_{N_{k_*}}u)|\sim N_{k_*-1}(C_{k_*}e^{-N_{k_*}^2t})^2.
}
We integrate in time to estimate that the total size of the perturbation is on the order of $N_{k_*-1}(C_{k_*}/N_{k_*})^2$ once $t\gtrsim N_{k_*}^{-2}$. Iterating this step as $k$ descends from $k_*$ down to $0$, and letting $C_k$ be the maximum amplitude of $P_{N_k}u$, we can expect the recurrence relation $C_k\sim N_k(C_{k+1}/N_{k+1})^2$ for all $k\in\{0,1,\ldots,k_*-1\}$. This leads to
\eq{
C_k\sim N_k\left(\frac{C_0}{N_0}\right)^{2^{-k}},\qquad k=0,1,\ldots,k_*.\label{recurrence-rough-intro-version-C}
}
The quantities $C_{k_*}/N_{k_*}$ and $C_0/N_0$ are essentially the $B^{-1}_{\infty,\infty}$ norms of $u(t)$ at $t=0$ and $t\sim N_0^{-2}$, respectively. Indeed, by time $N_0^{-2}$, higher frequency components of the solution have been made negligible by the dissipation. Thus, for any arbitrarily large choice of $C_0/N_0$, we can make $C_{k_*}/N_{k_*}=O(1)$
 by choosing a sufficiently large $k_*$. We emphasize that unlike in \cite{bourgain2008ill}, taking $k_*$ large is not costly for measuring the data in $B^{-1}_{\infty,1}$, etc., because there is always just a single mode present at $t=0$.

 Observe that the recurrence relation \eqref{recurrence-rough-intro-version-C} highlights that growth can only occur for large data ($C_0\gg N_0$), consistent with the Koch--Tataru theorem. A similar effect was observed in \cite{palasek2024non} and \cite{coic-palasek}.

\subsection{Notation, function spaces, and basic estimates}\label{definitions-notation-section}

We work on the 3D torus $\Td=(\mathbb R /2\pi\mathbb Z)^3$ and define $L^p(\Td)$ spaces in the standard way. When there is no ambiguity, we will often write $\|f\|_p$ in place of $\|f\|_{L^p(\Td)}$. Accordingly, if $f(x,t)$ is a function of time and space, we write $\|f(t)\|_p$ to mean the norm is taken only in space.

For $f\in L^2(\Td)$, we define the Fourier series and inverse,
\eqn{
\hat f(\xi)=\mathcal F f(\xi)\coloneqq \fint_\Td f(x)e^{-ix\cdot\xi}dx,&\qquad\xi\in\Zd\\
\mathcal F^{-1}f(x)\coloneqq\sum_{\xi\in\Zd}\hat f(\xi)e^{ix\cdot\xi},&\qquad x\in \Td
}
where $\fint_\Td=(2\pi)^{-3}\int_{\Td}$.

Let\footnote{While we make frequent use of the operators $P_N$, we very rarely refer to the symbols $\varphi$ and $\psi$. Thus, there should be no confusion with $\varphi_j$ and $\psi_k$ defined in Definitions~\ref{mikado-definition} and \ref{v-psi-etc-definition}, which always carry an index.} $\tilde\varphi\in C_c^\infty([0,\frac34))$ and $\tilde\psi\in C_c^\infty(\frac23,\frac32)$ be a partition of unity on $[0,\infty)$. Define the radial functions $\varphi(x)=\tilde\varphi(|x|)$ and $\psi(x)=\tilde\psi(|x|)$. For each dyadic number $N\in2^\mathbb N=\{1,2,4,8,\ldots\}$, we define the Littlewood--Paley projection $P_N$ as Fourier multiplier
\eqn{
P_Nf(x)\coloneqq \mathcal F^{-1}(\psi(\xi/N)\hat f(\xi))(x).
}
Variants of $P_N$ are defined in the obvious way; for instance, $P_{>N}\coloneqq\sum_{M>N}P_M$, etc. Here and elsewhere, sums over un-indexed capital letters $N$ and $M$ are implicitly taken over the dyadic numbers $2^\mathbb N$. The un-indexed dyadic numbers $N,M\in2^\mathbb N$ should not be confused with the distinguished sequences of indexed frequency scales $N_k,M_k$ defined in \S\ref{principal-construction-section}.

We freely make use of standard estimates for Littlewood--Paley projections and other well-known Fourier multipliers including the heat propagator $e^{t\Delta}$ and Leray projection $\mathbb P$. We refer to \cite[Appendix C]{coic-palasek} for a careful treatment of Littlewood--Paley theory and other standard estimates for Fourier multipliers on the torus.

We define the Besov spaces $B^s_{p,q}$ as distributions on $\Td$ for which the norm
\eqn{
\|u\|_{B^{s}_{p,q}}\coloneqq \big\|N^s\|P_Nu\|_{L^p(\Td)}\big\|_{\ell^q(2^\mathbb N)}
}
is finite. Once again, we will be measuring fields on $\Td$ with zero average so we neglect the distinction between $B^s_{p,q}$ and $\dot B^s_{p,q}$. We record the embedding
\eq{\label{besov-embedding}
B^s_{p,q_1}\subset B^s_{p,q_2}
}
when $q_1\leq q_2$, which is immediate from the corresponding embedding in the scale $\ell^q(\mathbb N)$. An alternative definition, which is enlightening but not used in practice in this work, is $\|f\|_{B^s_{p,q}}\sim_{s,p,q}\|t^{-\frac s2}e^{t\Delta}f\|_{L^q((0,\infty),t^{-1}dt;L^p)}$. It is easy to see that in the context of zero-average functions on the torus, $B_{\infty,\infty}^{n+\alpha}$ corresponds to the H\"older space $C^{n,\alpha}$ for $n\geq0$ an integer and $\alpha\in(0,1)$.

We give two equivalent definitions of $BMO^{-1}$. The standard definition is that a scalar function $f$ is in $BMO^{-1}$ if $f=\div g$, where $g$ is a vector field with $BMO$ components. The norm of $f$ is then defined as the minimal norm of $g$ in $BMO$. Here we define $BMO(\Rd)$ as the space of locally $L^1$ functions (modulo addition of a constant) satisfying
\eqn{
\|g\|_{BMO}&\coloneqq\sup_{Q}\fint_Q\left|g(x)-\Big(\fint_Qg\Big)\right|dx,
}
the supremum being taken over cubes in $\Rd$. Then $BMO(\Td)$ is precisely the space of periodic functions that are in $BMO(\Rd)$ when regarded as functions on all of $\Rd$, with the norm being inherited. Note that in $\Rd$ one might consider only cubes below (say) unit side length, but in the periodic setting there is no significant distinction.

Alternatively, one can define
\eqn{
\|f\|_{BMO^{-1}}=\|e^{t\Delta}f\|_{X_\infty}
}
with $X_T$ as in \eqref{XT-koch-tataru-definition}. Combining the heat flow formulation of the Besov spaces with this formulation of $BMO^{-1}$, one can verify the embedding
\eq{\label{besov-bmo-inclusion}
B^{-1}_{\infty,2}\subset BMO^{-1}\subset B^{-1}_{\infty,\infty}.
}
We also define the space $VMO^{-1}$ which is equipped with the same norm as $BMO^{-1}$, but contains only the closure of the test functions in that norm.

Finally, we make a few comments about notation. We denote the Euclidean basis in $\Rd$ by $\{e_1,e_2,e_3\}$. For $x\in\Rd$, we write $\langle x \rangle\coloneqq (1+|x|^2)^\frac12$. For $x,y\in\Rd$, we write $x\odot y\coloneqq \frac12(x\otimes y+y\otimes x)$, leading to the product rule $(x+y)\otimes(x+y)=x\otimes x+y\otimes y+2x\odot y$. If $X$ is any condition, we write the indicator function of $X$ as $\mathbf1_X$.

The notation $X\lesssim Y$ indicates that $X\leq CY$ for an absolute constant $C>0$, and analogously for $\gtrsim$. We write $X\sim Y$ to mean both $X\lesssim Y$ and $X\gtrsim Y$. If $X>0$, $O(X)$ stands in for any quantity bounded by $CX$ with $C>0$ an absolute constant. We adorn $\lesssim$, $O$, etc.\ with subscripts to indicate that the implicit constants may depend on the subscripted quantities.

The plan of the paper is as follows: the heart of the paper is \S\ref{second-theorem-proof-section} where we detail the construction leading to Theorem~\ref{second-theorem}. In \S\ref{principal-construction-section} we build the leading part of the solution; in \S\ref{full-solution-correction-subsection} we show how to perturb it to create the global solution $u(t,x)$; finally in \S\ref{final-estimates-conclusion-subsection} we conclude Theorem~\ref{second-theorem}. From there, Corollaries~\ref{a-priori-estimate-corollary} and \ref{kt-vmo-bmo-corollary} are straightforward and proved in \S\ref{corollary-section}. Appendix~\ref{mikado-appendix} contains some details of the geometry of the building blocks for the construction in \S\ref{principal-construction-section}.

\subsection*{Acknowledgments}

This work was supported by the National Science Foundation under Grant No.\ DMS-2424441.

\section{Proof of Theorem~\ref{second-theorem}}\label{second-theorem-proof-section}

\subsection{Construction of the principal part}\label{principal-construction-section}

We begin by defining several parameters that appear in the construction. Let $N_k=\big\lceil |\eta_*|A^{b^k-1}\big\rceil$ for some $b\in(1,2)$ and large $A>1$ to be specified. It will suffice to take, for instance, $b=3/2$. The $N_k$ are frequency scales, defining the oscillation of the components of the solutions. Let us point out the special case $N_0=\lceil|\eta_*|\rceil$.

We also fix a parameter $\gamma\in(b^{-1},1)$ obeying, more specifically,
\eqn{
\frac{b+1}{2b}<\gamma<\frac{5-b}4.
}
For each $k\geq1$, we define $M_k=\big\lceil  |\eta_*|A^{\gamma b^k-1}\big\rceil$. The length scales $M_k^{-1}$ capture the spatial localization of the component of the solution in the $N_k$ frequency shell. (Note that $M_0$ will not appear in the construction.)

We record the fact that
\eq{\label{M-over-N-fact-small}
M_k/N_k\sim A^{-(1-\gamma)b^k}\leq A^{-(1-\gamma)b}\qquad\forall k\geq1
}
which can be made as small as desired by choosing $A$ large. Arranging $M_k\ll N_k$ simplifies the estimates in two ways: it ensures that the solution is divergence-free to leading order, and that when the solution evolves under heat flow, the effect on the spatial localization is negligible. 

The projection of the velocity to frequencies near $N_k$ grows to a maximum amplitude which we denote by $C_k$, defined to solve the recursive relation
\eq{\label{C_relation}
C_{k+1}^2=N_k^{-1}N_{k+1}^{2}C_k,\quad C_0=N_0\frac{|\Theta_*|}{|\eta_*|}.
}
It follows that $C_k$ is given precisely by
\eq{\label{C_formula}
C_k=N_k\left(\frac{|\Theta_*|}{|\eta_*|}\right)^{2^{-k}}.
}

We also fix a large $k_*\in\mathbb N$, the number of active scales in the construction. The exact value will be specified later in order to ensure $C_{k_*}=O(1)$.

\begin{remark}
    We briefly explain the order in which the key quantities are chosen. First, $\gamma$ and $b$ can be chosen arbitrarily to satisfy the above relations. The targeted scales $\Theta_*$, $\eta_*$, and $\epsilon_*$ are fixed in the statement of the theorem. The number of scales $k_*$ is chosen large depending only on $|\Theta_*|$. Finally, $A$ is taken large depending on all of the other parameters.

    We also remark that at each stage in the construction, it is possible to control only a finite number of derivatives. Indeed, the implicit constants grow with the order of differentiation, and $A$ needs to be chosen large enough to defeat them all. Thus, in order to control $n_{max}$ derivatives of the error as claimed in Theorem~\ref{second-theorem}, it is necessary to keep track of many more derivatives over the course of the proof. This explains the role of the quantities $n_{max}'>n_{max}''>n_{max}'''>n_{max}$ appearing in Lemma~\ref{construct_coefficients_lemma} and Propositions~\ref{f-prop}--\ref{w-prop}.
\end{remark}

\subsubsection{The oscillatory part}

We make use of the following rank-one decomposition of symmetric positive definite tensors. Similar decompositions are standard in the convex integration literature; see for instance \cite{de2009euler}.

\begin{lem}[Nash lemma]\label{nash-lemma}
    Define the ball $B=B(\Id,\frac1{7})$ centered at the identity in the space of real symmetric $3\times3$ matrices, equipped with the entry-wise maximum norm. There exist $\theta_1,\ldots,\theta_6\in\mathbb Q^3\cap\mathbb S^2$ and coefficient functions $\Gamma_1,\ldots,\Gamma_6\in C^\infty(B;[\frac1{5},1])$ such that
\eq{\label{nash-lemma-span-equation}
M=\sum_{j=1}^6\Gamma_j(M)^2\theta_j\otimes \theta_j\qquad \forall M\in B.
}
Further, $\Gamma_j$ obey the uniform bounds
\eq{\label{gamma-bounds}
\|\grad^m\Gamma_j\|_{L^\infty(B)}\lesssim_m1
}
for all $m\geq1$.
\end{lem}

The proof is from an explicit definition\footnote{We point out that the parameters $\frac17$ and $\frac15$ appearing in the statement could be chosen in many different ways; this particular choice is favorable (but certainly not optimal) in its contribution to the estimated size of the data.} of $(\theta_j)_{j=1,\ldots,6}$ and $(\Gamma_j)_{j=1,\ldots,6}$ which we give in Appendix~\ref{mikado-appendix}.

Having fixed the flow directions $\theta_j$, we now define the geometry of the Mikado flow building blocks.

\begin{define}\label{mikado-definition}
    Let $\theta_j$ ($j=1,2,\ldots,6$) be as in Lemma~\ref{nash-lemma}. We define the following:
    \begin{itemize}
        \item oscillation directions $(\eta_j)_{j=1}^6$ chosen from\footnote{This is clearly possible; see the proof of Lemma~\ref{nash-lemma} in Appendix~\ref{mikado-appendix}.} $\{e_1,e_2,e_3\}$ such that $\eta_j\cdot\theta_j=0$,
        \item for each $k\in\{1,2,\ldots,k_*-1\}$, a standard mollifier $\phi_k$ at length scale
        \eqn{
        \ell_k\coloneqq N_k^{-\frac34}N_{k+1}^{-\frac14},        
        }
        \item Mikado flow positions $(x_j)_{j=1}^6\in\Td$, radius $\delta_0>0$, lines\footnote{By selecting $\theta_j\in\mathbb Z^d$, we ensure that the lines $\mathcal L_j$ are indeed $2\pi$-periodic; this is clearly essential for \eqref{mikado-separation-disjoint-condition}.} $\mathcal L_j=x_j+\mathbb R\theta_j\bmod2\pi\mathbb Z^3$, and cylinders $\mathcal C_j(r)=\{x\in\Td:\dist_{\Td}(x-\mathcal L_j)< r\}$ such that
        \eq{\label{mikado-separation-disjoint-condition}
        \dist_{\Td}(\mathcal L_{j_1},\mathcal L_{j_2})>\frac{201}{100}\delta_0\qquad\forall j_1\neq j_2,
        }
        \item Mikado flow profiles $(\varphi_j)_{j=1}^6\in C_c^\infty(\mathcal C_j(\delta_0))$ obeying the bounds
        \eq{
        0\leq\varphi_j\leq1,\qquad\|\grad^n\varphi_j\|_\infty\lesssim_n1,\label{mikado-pointwise-bounds}
        }
        the shear flow property
        \eq{\label{shear-flow-mikado-property}\theta_j\cdot\grad\varphi_j=0,}
        and the normalization conditions
        \eq{\varphi_j\equiv1\text{ on }\mathcal L_j\label{mikado-non-degeneracy-equals-one-1},\\
        \|\varphi_j\|_{2}^2= (10\pi^2-\epsilon_0)\delta_0^2\label{mikado-l2-normalization}
        }
        for all $j\in\{1,\ldots,6\}$ and any fixed $\epsilon_0\in(0,1)$.
    \end{itemize}
\end{define}

Note that there is tension between the pointwise upper bound in \eqref{mikado-pointwise-bounds} and the $L^2$ normalization \eqref{mikado-l2-normalization}. The volume of the cylinder $\mathcal C_j(\delta_0)$ restricted to one period of $\Td$ is $10\pi^2\delta_0^2$ (because $5\theta_j\in\mathbb Z^3$; see Appendix~\ref{mikado-appendix}), so \eqref{mikado-l2-normalization} can be achieved by setting, for instance, $\varphi_j\equiv1$ on nearly its whole support.

In order to make the constants explicit, we make the following concrete choice of $\delta_0$. The proof will be deferred to Appendix~\ref{mikado-appendix}.

\begin{lem}\label{mikado-delta-lemma}
    There exists a set of positions $(x_j)_{j=1}^6$ as in Definition~\ref{mikado-definition} such that we may take $\delta_0=1/15$.
\end{lem}

Finally we can define $v(x,t)$, the leading part of the solution.

\begin{define}\label{v-psi-etc-definition}
    For each $k\in\{1,2,\ldots,k_*-1\}$ and $j\in\{1,2,\ldots,6\}$, we define the oscillatory profiles
    \begin{align*}
    \Psi_{j,k}(x,t) &= \varphi_{j}(M_k x) \sin(N_k\eta_j\cdot x)e^{-N_k^2t}
    \end{align*}
    from which we build the vector potentials\footnote{Here and elsewhere, we write $\sum_j$ to mean the sum over $j\in\{1,2\ldots,6\}$.}
    \begin{align*}
        \psi_k(x,t)&=N_k^{-2}\phi_k*\sum_ja_{j,k}(x)\Psi_{j,k}(x,t)\theta_j
    \end{align*}
    for some coefficients $a_{j,k}(x)$ to be specified. Then the velocity field is defined as
    \eqn{
    v(x,t)&=\sum_{k=0}^{k_*}v_k(x,t)
    }
    where $v_k(x,t)$ are the frequency localized components given by
    \begin{align*}
v_k(x,t)&=C_k(1-e^{-2N_{k+1}^2t})\mathbb P\Delta \psi_k(x,t).
\end{align*}
Now we define the two exceptional cases: first, when $k=0$, we instead have
\eqn{
\Psi_{j,0}(x,t)&=\sin(x\cdot\eta_*)e^{-|\eta_*|^2t},\\
\psi_0(x,t)&=N_0^{-2}\sum_ja_{j,0}(x)\Psi_{j,0}(x,t)\frac{\Theta_*}{|\Theta_*|}
}
with $v_0(x,t)$ following the pattern above; second, when $k=k_*$, we instead have
\eqn{
\psi_{k_*}(x,t)&=N_{k_*}^{-2}\sum_ja_{j,k_*}(x)\Psi_{j,k_*}(x,t)\theta_j,\\
v_{k_*}(x,t)&=C_{k_*}\mathbb P\Delta\psi_{k_*}(x,t)
}
with $\Psi_{j,k_*}(x,t)$ following the same pattern as $k\in\{1,2,\ldots, k_*-1\}$.

We sometimes write $\Psi_{j,k}^0$, $\psi_k^0$, $v^0$, and $v_k^0$ to refer to the initial ($t=0$) value of $\Psi_{j,k}$, $\psi_k$, $v$, and $v_k$, respectively.
\end{define}

\begin{remark}\label{new-remark}

Some comments on the rationale behind Definition~\ref{v-psi-etc-definition} are in order.

\begin{enumerate}[1.]
    \item The $(1-e^{-2N_{k+1}^2t})$ factor in the definition of $v_k$ reflects the activation of the $k$th scale and growth to maximum amplitude $C_k$ on time scale $N_{k+1}^{-2}$. Note that the factor is missing in the definition of $v_{k_*}$, which is not activated by instead \emph{begins} at amplitude $C_{k_*}$ at time $t=0$.
    \item Note that the mollification in the definition of $\psi_k$ is essentially negligible, the length scale being much smaller than $N_k^{-1}$. It is included only to avoid a loss of derivatives problem in the proof of Lemma~\ref{construct_coefficients_lemma}.
    \item We point out some elementary but essential consequences of Definition~\ref{v-psi-etc-definition} that will inform the computations in \S\ref{full-solution-correction-subsection}. Because of $\theta_j\cdot\eta_j=0$ and \eqref{shear-flow-mikado-property}, we have
\eqn{
\div\psi_k(x,t)=N_k^{-2}\phi_k*\sum_j(\theta_j\cdot\grad a_{j,k})(x)\Psi_{j,k}(x,t),
}
with appropriate adjustments when $k\in\{0,k_*\}$. Derivatives of $a_{j,k}$ should be thought of as lower order, so $\psi_k$ is nearly divergence-free. For the same reason, and additionally invoking \eqref{mikado-separation-disjoint-condition}, we have
\eqn{
\div(\Psi_{j,k}\theta_j)\otimes (\Psi_{j,k}\theta_j)=0
}
identically. Thus, up to similar lower order errors, $\psi_k(x,t)$ are shear flows. Finally, observe that $\Psi_{j,k}(x,t)$ solves the heat equation to leading order, with error terms involving derivatives of $\varphi_j$.
\end{enumerate}

\end{remark}

\subsubsection{The coefficients and some estimates}

For $j\in\{1,2,\ldots,6\}$ and $k\in\{1,2,\ldots,k_*\}$, define the quantity
\eqn{
B_{j,k}\coloneqq\fint_{\mathbb T^3}\Psi_{j,k}^2(x,0)dx
}
which measures the energy transfer rate of the $j$th Mikado flow at scale $k$. Using, for example, the improved H\"older's inequality in \cite[Lemma 2.1]{modena2018non} with $p=1$, we can decouple the low- and high-frequency factors in the integral to find
\eqn{
B_{j,k}&=\fint_\Td\varphi_j^2dx\fint_\Td\sin^2(x_1)dx+O(N_k^{-1}M_k).
}
By the observation \eqref{M-over-N-fact-small}, the second term can be made as small as needed by the choice of $A$. Further, by \eqref{mikado-l2-normalization} and Lemma~\ref{mikado-delta-lemma}, the leading term is exactly $1/(360\pi)-O(\epsilon_0)$ where $360\pi\approx 1130.97$. Thus, by the choice of $A$ and $\epsilon_0$, we can ensure
\eq{\label{B-sim-1}
B_{j,k}\in\Big[\frac1{1131},\frac1{1130}\Big]
}
for all $k\geq1$.

Next, define the modified symmetric gradient
\eqn{
\mathcal Df\coloneqq\frac12(\grad f+\grad f^T)-(\div f)\Id
}
for $f\in C^\infty(\Td;\Rd)$, which obeys
\eq{\label{d-property}
\div\mathcal Df=\frac12\mathbb P\Delta f=-\frac12\curl\curl f.
}

We now define the coefficients $a_{j,k}(x)$ that appear in Definition~\ref{v-psi-etc-definition}. The central point of the construction is that by imposing the identity \eqref{a-recursive-relation}, the inverse cascade initiated by $v_{k_*}^0$ will unfold in a precisely controlled way.

\begin{lem}\label{construct_coefficients_lemma}
    Let $\psi_k$, $\theta_j$, etc.\ be as in Lemma~\ref{nash-lemma} and Definition \ref{v-psi-etc-definition}. There exist choices of $a_{j,k}\in C^\infty(\mathbb T^3;\mathbb R)$ such that
    \eqn{
    a_{j,0}=-\frac{N_0}{|\eta_*|}\delta_{j,1}
    }
    with $\delta$ being the Kronecker delta, while for the remaining $k$ we have the inductive formula
    \eq{\label{a-recursive-relation}
\sum_jB_{j,k}a_{j,k}^2(x)\theta_j\otimes\theta_j=-4N_{k-1}\mathcal D\psi_{k-1}^0(x)+p_{k}\Id\quad\forall x\in\mathbb T^3
}
for all $k\in\{1,2\ldots,k_*\}$ and some constants $p_k\in\mathbb R$. Furthermore, for any choice of $n_{max}'\in\mathbb N$, $A$ can be chosen large enough so that
\begin{align}
    1\leq a_{j,k}(x)\leq 32000&\quad\forall x\in\Td\label{a-precise-upper-lower-bounds-final}\\
    \|\grad^n a_{j,k}\|_\infty\lesssim_nN_{k-1}^n&,\label{final-a-bounds}\\
    \|\grad^n\psi_k(t)\|_\infty\lesssim N_k^{-2+n}e^{-N_k^2t}&\quad\forall t\geq0\label{psi-bound-final}
\end{align}
for all $j\in\{1,\ldots,6\}$, $n=0,1,2,\ldots,n_{max}'$, and $k\geq1$.
\end{lem}

\begin{proof}
    With $\Gamma_j$ as in Lemma~\ref{nash-lemma}, we inductively define
    \eqn{
    a_{j,k}(x)\coloneqq \left(28B_{j,k}^{-1}N_{k-1}\|\mathcal D\psi_{k-1}^0\|_\infty\right)^\frac12\Gamma_j(\Id-\frac{\mathcal D\psi_{k-1}^0}{7\|\mathcal D\psi_{k-1}^0\|_\infty})
    }
    for $k\in\{1,\ldots,k_*\}$. We will see below that, with this definition, $\|\mathcal D\psi_{k-1}^0\|_\infty$ is strictly positive for all such $k$. Furthermore, $\Id-\frac{\mathcal D\psi_{k-1}^0}{7\|\mathcal D\psi_{k-1}^0\|_\infty}$ is a symmetric matrix in $B(\Id,1/7)$; thus $\Gamma_j$ is well-defined as it appears in the definition.
    
    From \eqref{nash-lemma-span-equation}, it is straightforward to see that this choice of $a_{j,k}$ satisfies \eqref{a-recursive-relation}. Further, using \eqref{mikado-pointwise-bounds}, \eqref{B-sim-1}, and the fact that $\Gamma_j$ maps into $[\frac15,1]$, we obtain the following precise estimate for $a_{j,k}$:
    \eq{\label{a-initial-precise-estimate}
    \big(100N_{k-1}\|\mathcal D\psi_{k-1}^0\|_\infty\big)^\frac12\leq a_{j,k}(x)\leq \big(31668N_{k-1}\|\mathcal D\psi_{k-1}^0\|_\infty\big)^\frac12
    }
    for all $x\in\Td$ and $k\in\{1,\ldots,k_*\}$. The derivative on $\psi_k^0$ can fall in three places: the coefficient $a_{j,k}$, the cutoff $\Psi_{j,k}$, and the sin factor. Expanding with the Leibniz rule and applying \eqref{mikado-pointwise-bounds}, \eqref{B-sim-1}, \eqref{gamma-bounds}, and \eqref{M-over-N-fact-small}, we obtain
    \eqn{
    \|\mathcal D\psi_k^0\|_\infty&\lesssim N_k^{-1}\|a_{j,k}\|_\infty + N_k^{-2}\|\grad a_{j,k}\|_\infty.
    }
    It is similarly straightforward to obtain
    \eq{\label{gradabound}
    \|\grad a_{j,k}\|_\infty&\lesssim N_{k-1}^\frac12\frac{\|\grad\mathcal D\psi_{k-1}^0\|_\infty}{\|\mathcal D\psi_{k-1}^0\|_\infty^{1/2}}.
    }
    Combining these with \eqref{a-initial-precise-estimate}, we arrive at the inequality
    \eq{
    \|\mathcal D\psi_k^0\|_\infty&\lesssim N_k^{-1}N_{k-1}^\frac12\|\mathcal D\psi_{k-1}^0\|_\infty^\frac12\left(1+N_k^{-1}\frac{\|\grad\mathcal D\psi_{k-1}^0\|_\infty}{\|\mathcal D\psi_{k-1}^0\|_\infty}\right)
    }
    for all $k\in\{1,2,\ldots,k_*\}$. From here we claim the more precise bound
    \eq{\label{Dpsi_upper_and_lower_bounds}
    c_*^{2^{-k}}N_k^{-1}\leq \|\mathcal D\psi_k^0\|_\infty\leq 31669^{1-2^{-k}}N_k^{-1}
    }
    for a $c_*\in(0,1]$ to be specified. We induct from $k=0$, in which case we have the exact identity
    \eqn{
    \mathcal D\psi_0^0=-N_0^{-1}\frac{\eta_*}{|\eta_*|}\odot\frac{\Theta_*}{|\Theta_*|}\cos(x\cdot\eta_*)
    }
    which leads to
    \eq{\label{k-equals-zero-case-induction}
    N_0\|\mathcal D\psi_0^0\|_\infty=\Big\|\frac{\eta_*}{|\eta_*|}\odot\frac{\Theta_*}{|\Theta_*|}\Big\|\eqcolon c_*.
    }
    Based on the fact that $\frac{\eta_*}{|\eta_*|}$ and $\frac{\Theta_*}{|\Theta_*|}$ are orthogonal unit vectors (and recalling that we measure matrices with the entry-wise maximum), it is simple to show\footnote{If $a,b$ are orthogonal unit vectors, then we compute $\sum_{i,j}(a_ib_j+b_ia_j)^2=2$; by the pigeonhole principle, there is an $(i,j)$ with $|a_ib_j+b_ia_j|/2\geq1/(3\sqrt2)$.} $c_*\in[(3\sqrt2)^{-1},1]$.
    
    To continue the induction, fix $k\in\{1,\ldots,k_*\}$ and assume \eqref{Dpsi_upper_and_lower_bounds} for all indices up to $k-1$. Letting the derivatives fall on the mollifier and applying \eqref{a-initial-precise-estimate} and \eqref{Dpsi_upper_and_lower_bounds}, we find
    \eqn{
    \|\grad\mathcal D\psi_{k-1}^0\|_\infty&\lesssim N_{k-1}^{-2}\ell_{k-1}^{-2}N_{k-2}^\frac12\|\mathcal D\psi_{k-2}^0\|_\infty^\frac12\lesssim N_{k-1}^{-2}\ell_{k-1}^{-2}.
    }
    From this, \eqref{a-initial-precise-estimate}, \eqref{gradabound}, and \eqref{Dpsi_upper_and_lower_bounds}, it follows that
    \eq{\label{a-and-grad-a-estimate}
    \|a_{j,k}\|_\infty\lesssim 1,\qquad\|\grad a_{j,k}\|_\infty\lesssim N_{k-1}^{-1}\ell_{k-1}^{-2}.
    }
    Toward showing \eqref{Dpsi_upper_and_lower_bounds}, we identify the leading contribution to $N_k\mathcal D\psi_k^0$:
    \eq{\label{ndpsi-expansion}
    N_k\mathcal D\psi_k^0=\sum_ja_{j,k}(x)\varphi_j(M_kx)\cos(N_k\eta_j\cdot x)\eta_j\odot\theta_j+E_1+E_2
    }
    where the errors
    \eqn{
    E_1&=(\phi_k-\delta)*\sum_ja_{j,k}(x)\varphi_j(M_kx)\cos(N_k\eta_j\cdot x)\eta_j\odot\theta_j
    }
    and
    \eqn{
    E_2&=N_k^{-2}\mathcal D\phi_k*\sum_j\theta_j\,\odot'\,\grad(a_{j,k}\varphi_j(M_kx))\sin(N_k\eta_j\cdot x)
    }
    capture the effect of removing the mollifier (with $\delta$ the Dirac delta) and terms with lower frequency derivatives, respectively. Note that we have used the temporary notation $a\,\odot'\,b\coloneqq \frac12(a\otimes b+b\otimes a)-(a\cdot b)\Id$ for a modified symmetric tensor product.
    
    By standard mollifier estimates, \eqref{mikado-pointwise-bounds}, \eqref{M-over-N-fact-small}, \eqref{Dpsi_upper_and_lower_bounds}, and \eqref{a-and-grad-a-estimate},
    \eqn{
    \|E_1\|_\infty&\lesssim \ell_k\sum_j\|\grad(a_{j,k}\varphi_j(M_kx)\cos(N_k\eta_j\cdot x))\|_\infty\\
    &\lesssim \max_j\,\ell_k(N_k\|a_{j,k}\|_\infty+\|\grad a_{j,k}\|_\infty)\\
    &\lesssim \ell_k(N_k+N_{k-1}^{-1}\ell_{k-1}^{-2})\lesssim (N_k/N_{k+1})^\frac14
    }
    and, similarly,
    \eqn{
    \|E_2\|_\infty&\lesssim \ell_k^{-1}N_k^{-2}(N_{k-1}^{-1}\ell_{k-1}^{-2}+M_k)\lesssim \ell_k^{-1}N_k^{-2}M_k.
    }
    In the last inequality, we have used $\gamma>\frac{1+b}{2b}$ to infer that the second term dominates. For the same reason, we see that the upper bound on $E_2$ dominates that on $E_1$.
    
    Now we estimate the leading term. By \eqref{mikado-pointwise-bounds}, \eqref{mikado-non-degeneracy-equals-one-1}, and \eqref{a-initial-precise-estimate},
    \eqn{
    \|a_{j,k}(x)\varphi_j(M_kx)\cos(N_k\eta_j\cdot x)\eta_j\odot\theta_j\|_\infty\in\big[\frac2{5}\inf_\Td|a_{j,k}|,\sup_\Td|a_{j,k}|\big],
    }
    the $2/5$ factor being the explicitly computable value of $\|\eta_j\odot\theta_j\|$ for all $j$. Combining this with \eqref{ndpsi-expansion} and the estimates for $E_1$ and $E_2$,
    \eqn{
    N_k\|\mathcal D\psi_k^0\|_\infty&\in\big[\frac3{10}\inf_{j,x}|a_{j,k}| ,\sup_{j,x}|a_{j,k}|\big] + O(M_kN_k^{-\frac54}N_{k+1}^\frac14)\\
    &\subset[(5N_{k-1}\|\mathcal D\psi_{k-1}^0\|_\infty)^{1/2},(31668N_{k-1}\|\mathcal D\psi_{k-1}^0\|_\infty)^{1/2}]+O(A^{-(\frac{5-b}4-\gamma)})
    }
    where the notation $[a,b]+O(y)$ is an abbreviation for the interval $[a-O(y),b+O(y)]$. Note that by the inductive hypothesis \eqref{Dpsi_upper_and_lower_bounds}, $(N_{k-1}\|\mathcal D\psi_{k-1}^0\|_\infty)^{1/2}$ is uniformly bounded from above and below; thus by taking $A$ sufficiently large, we can make the error much smaller to conclude
    \eqn{
 N_k\|\mathcal D\psi_k^0\|_\infty\in[(N_{k-1}\|\mathcal D\psi_{k-1}^0\|_\infty)^\frac12,(31669N_{k-1}\|\mathcal D\psi_{k-1}^0\|_\infty)^\frac12].
    }
    Iterating from the initial bound \eqref{k-equals-zero-case-induction} yields \eqref{Dpsi_upper_and_lower_bounds}, completing the induction. Combining with \eqref{a-initial-precise-estimate}, the claim \eqref{a-precise-upper-lower-bounds-final} is immediate.

    Having established \eqref{a-precise-upper-lower-bounds-final} and \eqref{Dpsi_upper_and_lower_bounds}, we can continue the procedure to bootstrap the higher order bounds in \eqref{final-a-bounds}--\eqref{psi-bound-final}: one obtains a rudimentary estimate on $\grad^{m+1}\psi_{k-1}^0$ by placing derivatives on the mollifier, then uses it to obtain a rudimentary estimate on $\grad^ma_{j,k}$, which in turn leads to the optimal estimate on $\grad^m\psi_k^0$. The key point is that when repeatedly differentiating $\psi_k^0$, the leading contribution occurs when all the derivatives fall on $\sin(N_k\eta_j\cdot x)$ (even using just the ``rudimentary bound'' on $a_{j,k}$). 
    
    Note that we insist $n$ is bounded in \eqref{final-a-bounds}--\eqref{psi-bound-final} because the errors from derivatives hitting $a_{j,k}$ grow with $n$, and $A$ has to be chosen large to compensate.
\end{proof}

Collecting the definitions of the building blocks and coefficients, we can derive some straightforward estimates on the principal part $v$ of the solution.

\begin{lem}\label{v-estimates-lemma}
    With $v_k$, $v$, and $n_{max}'$ as in Definition~\ref{v-psi-etc-definition}, we have
    \eq{\label{vk_bounds}
    \|\grad^nv_k(t)\|_\infty&\lesssim_n \left(\frac{|\Theta_*|}{|\eta_*|}\right)^{2^{-k}}N_k^{1+n}e^{-N_k^2t}\quad\forall t\geq0
    }
   for all $k\geq0$ and $n\leq n_{max}'$, and
   \eq{\label{critical_v_bounds}
    \|t^{\frac12-\frac1p}v\|_{L_t^p([0,\infty);L_x^\infty(\Td))}\lesssim_{\Theta_*,k_*}1
    }
    for all $p\in[1,\infty]$.
\end{lem}

\begin{proof}
    The proof of \eqref{vk_bounds} is immediate from \eqref{C_formula} and \eqref{psi-bound-final}, recalling that $\mathbb P\Delta=-\curl\curl$ so we can easily distribute the derivatives using the Leibniz rule.
    
    Now we turn to \eqref{critical_v_bounds} and use \eqref{vk_bounds} to very crudely bound
    \eqn{
    \|t^{\frac12-\frac1p}v\|_{L_t^pL_x^\infty}&\lesssim_{\Theta_*}\sum_{k=0}^{k_*}\|t^{\frac12-\frac1p}N_ke^{-N_k^2t}\|_{L^p_t}\lesssim k_*,
    }
    each summand being $O(1)$.
\end{proof}

\subsection{Construction of the full solution}\label{full-solution-correction-subsection}

In \S\ref{principal-construction-section}, we defined only the leading part $v(x,t)$ of the solution. What remains is to show that it solves \eqref{nse} up to a small error, and can therefore be perturbed into an exact solution $u(x,t)$.

\begin{define}\label{full-definition}
Define the initial data
\eqn{
u^0(x)\coloneqq v(x,0)\in C^\infty(\Td)
}
where $v$ is as constructed in \S\ref{principal-construction-section}. Then let $u$ be the local-in-time smooth solution of \eqref{nse} from $u^0$. Let $T_*\in(0,\infty]$ be the maximum time of existence.

For $t\in[0,T_*)$, define
\eqn{
w(t,x)=u(t,x)-v(t,x).
}
\end{define}

By virtue of its definition, $w$ will solve the Navier--Stokes equations with a force and additional linear terms coming from $v$. Long-time regularity and smallness of $w$ requires the force to be designed carefully so that it admits sharp subcritical estimates, which we prove in the next proposition.

\begin{proposition}\label{f-prop}
    Let $\alpha,\beta\in(0,\frac12)$ be sufficiently small, depending only on $b$ and $\gamma$. For any $n_{max}''\in\mathbb N$ and all sufficiently large $A=A(n_{max}'')$, there exists a force $f\in C^\infty([0,\infty)\times\Td;\mathbb R^{3\times3})$ admitting the estimates
    \eq{\label{final-f-bound}
    t^{\frac12(2+n+\alpha)}\|\grad^nf(t)\|_{C^\alpha}&\lesssim_{\Theta_*,n} A^{-\beta}(N_0^2t)^{\frac\alpha2+\beta}e^{-N_0^2t}\qquad\forall t>0
    }
    for $n=0,1,2,\ldots,n_{max}''$, such that $v$ obeys
    \eq{\label{v_equation}
    \d_tv-\Delta v+\mathbb P\div v\otimes v=\mathbb P\div f.
    }
\end{proposition}

\begin{proof}
We decompose $v_k=v_k^p+v_k^e$ where
\eqn{
v_k^p(x,t)\coloneqq -C_k\sum_ja_{j,k}(x)\Psi_{j,k}(x,t)\theta_j
}
with the modification that $\theta_j$ is replaced by $\Theta_*/|\Theta_*|$ when $k=0$. Further, for $k\in\{1,\ldots,k_*\}$, we have the decomposition
\eqn{
\mathbb P\div v_k^p\otimes v_k^p=\mathbb P\div(R_k^{low}+R_k^{high})
}
where
\eqn{
R_k^{low}&=C_k^2\sum_ja_{j,k}^2B_{j,k}\theta_j\otimes\theta_je^{-2N_k^2t}
}
and
\eqn{
R_k^{high}&=C_k^2\mathcal R\sum_j(\Psi_{j,k}^2-B_{j,k}e^{-2N_k^2t})\div(a_{j,k}^2\theta_j\otimes\theta_j),
}
with $\mathcal R$ the anti-divergence operator defined in coordinates as $(\mathcal RV)_{ij}\coloneqq\mathcal R_{ijk}V_k$, where
\eqn{
\mathcal R_{ijk}=-\frac12\Delta^{-2}\d^3_{ijk}-\frac1{2}\Delta^{-1}\delta_{ij}\d_k+\Delta^{-1}\delta_{jk}\d_i+\Delta^{-1}\delta_{ik}\d_j.
}
To attain this decomposition, we have used two essential features of the Mikado flows. First, in the expansion of the product $v_k^p\otimes v_k^p$, there are no cross terms between distinct $j$'s because of the spatial separation between the supports, as guaranteed by \eqref{mikado-separation-disjoint-condition}. Second, the derivative identically vanishes when falling on $\Psi_{j,k}^2$ because of the shear flow property \eqref{shear-flow-mikado-property} as well as the fact that $\theta_j\cdot\eta_j=0$. (See Remark~\ref{new-remark}.)

The operator $\mathcal R$ first appeared (we believe) in \cite{de2009euler} and acts as a right-inverse of $\div$ for zero-average symmetric tensors:
\eq{\label{R-property}\div\mathcal R V(x)=V(x) -\fint_\Td V(y)dy,}
for any smooth vector field $V$.

Incrementing $k$ to $k+1$, we calculate using \eqref{a-recursive-relation} and \eqref{d-property} that
\eqn{
\mathbb P\div R_{k+1}^{low}&=-2C_{k+1}^2N_k\mathbb P\Delta\psi_k^0e^{-2N_{k+1}^2t}.
}
Next, we see from the Definition~\ref{v-psi-etc-definition} that
\eqn{
(\d_t-\Delta)v_k&=2C_kN_{k+1}^2e^{-2N_{k+1}^2t}\mathbb P\Delta\psi_k\mathbf1_{k\neq k_*}+\mathbb P\div E_{1,k}
}
where $E_{1,k}$ is the error from $-\Delta$ falling on the lower frequency factors, and from removing the mollifier. Explicitly, we have $E_{1,0}=0$ when $k=0$ (since $a_{j,0}$ are constant and there is no mollification);
\eqn{
E_{1,k}&=C_kN_k^{-2}(1-e^{-2N_{k+1}^2t})e^{-N_k^2t}\\
&\qquad\times\phi_k*\mathcal R\sum_j\sum_{\substack{|\alpha|+|\beta|=4\\|\beta|\neq4}}c_{\alpha,\beta}\d^\alpha (a_{j,k}\varphi_j(M_k x))\d^\beta(\sin(N_k\eta_j\cdot x))\theta_j\\
&\quad+2C_kN_{k+1}^2e^{-2N_{k+1}^2t}(\phi_k-\delta)*\mathcal D\psi_k
}
for $k\in\{1,2,\ldots,k_*-1\}$; and
\eqn{
E_{1,k_*}&=C_{k_*}N_{k_*}^{-2}e^{-N_{k_*}^2t}\mathcal R\sum_j\sum_{\substack{|\alpha|+|\beta|=4\\|\beta|\neq4}}c_{\alpha,\beta}\d^\alpha (a_{j,{k_*}}\varphi_j(M_{k_*} x))\d^\beta(\sin(N_{k_*}\eta_j\cdot x))\theta_j
}
for $k=k_*$. Here, $c_{\alpha,\beta}$ are combinatorial constants coming from expanding $-\Delta\mathbb P\Delta$ over the products. Note that, crucially, the terms with $|\beta|=4$ belong to the leading part which satisfies the heat equation exactly, so are not included in $E_{1,k}$.

Now we take advantage of the key cancellation, which is the culmination of the careful construction in Lemma~\ref{construct_coefficients_lemma}. In light of the relationship \eqref{C_relation} between $C_k$ and $C_{k+1}$, it follows that for $k=0,1,\ldots, k_*-1$, $\mathbb P\div R_{k+1}^{low}$ nearly cancels with the leading term in $(\d_t-\Delta)v_k$ leading to
\eqn{
(\d_t-\Delta)v_k+\mathbb P\div v_{k+1}^p\otimes v_{k+1}^p&=\mathbb P\div (E_{1,k}+E_{2,k}+R_{k+1}^{high}).
}
The new error term,
\eqn{
E_{2,k}=-C_{k+1}^2N_k(1-e^{-N_{k}^2t})e^{-2N_{k+1}^2t}\mathcal D\psi_k,
}
reflects the fact that $(\d_t-\Delta)v_k$ and $\mathbb P\div R_{k+1}^{low}$ differ by a factor of $e^{-N_k^2t}$.

To estimate $E_{1,k}$, we recall the following stationary phase estimate for $\mathcal R$; see for instance \cite[Corollary 5.3]{de2009euler}:
\eq{\label{R-stationary-phase}
\|\mathcal R(b(x)e^{i\lambda \xi\cdot x})\|_{C^\alpha}&\lesssim_{m,\alpha} \lambda^{-1+\alpha}\|b\|_\infty+\lambda^{-m+\alpha}\|\grad^mb\|_\infty+\lambda^{-m}\|\grad^mb\|_{C^\alpha}
}
for any $\xi\in\mathbb Z^3\setminus0$, $\lambda>0$, and smooth vector field $b$ on $\Td$. Further, it is clear from its definition as a multiplier that $\grad\mathcal R$ is a Calder\'on--Zygmund operator so it is bounded on $C^{n,\alpha}$ for $\alpha\in(0,1)$. 

With \eqref{R-stationary-phase} in hand, we have
\begin{equation}\begin{aligned}\label{E1_est}
E_{1,0}&\equiv0,\\
\|\grad^nE_{1,k}\|_{C^\alpha}
&\lesssim_{\Theta_*,n} (N_k^{1+n+\alpha}M_k+N_k^{2+n+\alpha}M_k\left(\frac{M_k}{N_k}\right)^m)e^{-N_k^2t}\\
&\qquad+\ell_kN_k^{1+n}N_{k+1}^2e^{-N_{k+1}^2t}\qquad\forall k\in\{1,2,\ldots,k_*\}
\end{aligned}\end{equation}
using \eqref{C_formula}, \eqref{final-a-bounds} with a sufficiently large choice of $n_{max}'$, \eqref{mikado-pointwise-bounds}, and \eqref{psi-bound-final}. Recalling \eqref{M-over-N-fact-small}, we can choose $m$ large depending on $\gamma$ to make $N_k^{2+\alpha}M_k(M_k/N_k)$ negligible compared to the first term on the right.

Next, using the pointwise bound $(1-e^{-N_k^2t})e^{-N_{k+1}^2t}\leq N_k^2N_{k+1}^{-2}$ along with \eqref{C_formula} and \eqref{psi-bound-final}, we estimate
\eq{\label{E2_est}
\|\grad^nE_{2,k}\|_\infty&\lesssim_{\Theta_*,n} N_{k}^{2+n}e^{-N_{k+1}^2t}.
}
Now consider the high frequency part of the interaction. We write
\eqn{
R_k^{high}&=C_k^2e^{-2N_k^2t}\sum_j\sum_{\xi\in\mathbb Z^3}\mathcal F((\Psi_{j,k}^0)^2-B_{j,k})(\xi)\mathcal R\left(e^{ix\cdot\xi}\div(a_{j,k}^2\theta_j\otimes\theta_j)\right)
}
and observe that $\mathcal F((\Psi_{j,k}^0)^2-B_{j,k})(\xi)=\mathcal F_x\Big(\varphi_j^2(x)\sin^2(M_k^{-1}N_k\eta_j\cdot x)\Big)(\xi/M_k)$ when $\xi\in M_k\mathbb Z^3\setminus0$ and vanishes otherwise. Using additionally that $\sin^2(M_k^{-1}N_kx\cdot\eta_j)$ is supported only at the three frequencies $\xi=0$ and $\xi=\pm 2M_k^{-1}N_k\eta_j$, we apply \eqref{C_formula} and \eqref{R-stationary-phase} to find
\eqn{
\|R_k^{high}\|_{C^\alpha}&\lesssim C_k^2e^{-2N_k^2t}\sum_j\sum_{\xi\in\mathbb Z^3\setminus0}\Big((|\widehat{\varphi_j^2}(\xi)|+|\widehat{\varphi_j^2}(\xi+2\frac{N_k}{M_k}\eta_j)|+|\widehat{\varphi_j^2}(\xi-2\frac{N_k}{M_k}\eta_j)|)\\
&\qquad\times\|\mathcal R\left(e^{iM_k x\cdot\xi}\div(a_{j,k}^2\theta_j\otimes\theta_j)\right)\|_{C^\alpha}\Big)\\
&\lesssim_{\Theta_*,m} N_k^{2}e^{-2N_k^2t}\sup_j\sum_{\xi\in\mathbb Z^3\setminus0}\Big((\langle\xi\rangle^{-5}+\langle\xi+2\frac{N_k}{M_k}\eta_j\rangle^{-5}+\langle\xi-2\frac{N_k}{M_k}\eta_j\rangle^{-5})\\
&\qquad\times (M_k^{-1+\alpha}N_{k-1}+M_k^{-1+\alpha}N_k^{-m}N_{k-1}^{1+m}+M_k^{-1}N_k^{-m}N_{k-1}^{1+\alpha+m})\Big)\\
&\lesssim N_k^{2}M_k^{-1+\alpha}N_{k-1}e^{-2N_k^2t},
}
having once again chosen $m$ large enough to suppress the extra terms. Note that the bounds on $\widehat{\varphi_j^2}$ follow from \eqref{mikado-pointwise-bounds}.

Similarly, we estimate the derivatives of $R_k^{high}$ using that $\grad\mathcal R$ is a Calder\'on--Zygmund operator:
\eqn{
\|\grad^n R_k^{high}\|_{C^\alpha}&\lesssim C_k^2e^{-2N_k^2t}\sum_j\sum_{\xi\in\mathbb Z^3}\Big((|\widehat{\varphi_j^2}(\xi)|+|\widehat{\varphi_j^2}(\xi+2\frac{N_k}{M_k}\eta_j)|+|\widehat{\varphi_j^2}(\xi-2\frac{N_k}{M_k}\eta_j)|)\\
&\qquad\times\|e^{iM_k x\cdot\xi}\div(a_{j,k}^2\theta_j\otimes\theta_j)\|_{C^{n-1,\alpha}}\Big)\\
&\hspace{-.3in}\lesssim_{\Theta_*} \sup_jN_k^{2}N_{k-1}e^{-2N_k^2t}\\
&\hspace{-.3in}\qquad\times\sum_{\xi\in\mathbb Z^3}(\langle\xi\rangle^{-5}+\langle\xi+2\frac{N_k}{M_k}\eta_j\rangle^{-5}+\langle\xi-2\frac{N_k}{M_k}\eta_j\rangle^{-5})(M_k|\xi|+N_{k-1})^{n-1+\alpha}\\
&\hspace{-.3in}\lesssim N_k^{1+n+\alpha}N_{k-1}e^{-2N_k^2t}
}
for all $n\geq1$ up to some fixed threshold that we can prescribe. Since $2-(1-\alpha)\gamma>1+\alpha$, we can increment $k\to k+1$ and combine these estimates into
\eq{\label{Rkhigh-final-estimate}
\|\grad^nR_{k+1}^{high}\|_{C^\alpha}&\lesssim N_{k+1}^{2+n}M_{k+1}^{-1+\alpha}N_{k}e^{-2N_{k+1}^2t}
}
for $n\geq0$ up to some threshold.

Collecting the remaining terms, we have \eqref{v_equation} with
\begin{equation}\begin{aligned}\label{full-f}
f&=\sum_{k=0}^{k_*}(E_{1,k}+E_{2,k}+R_{k+1}^{high}+v_k^e\otimes v_k^e+2v_k^e\odot v_k^p)\\
&\qquad+\sum_{k_1\neq k_2}v_{k_1}\otimes v_{k_2}+\mathcal R\div v_0^p\otimes v_0^p,
\end{aligned}\end{equation}
assigning $R_{k_*+1}^{high}\equiv0$. We claim $v_k^e$ is small; indeed, it has the exact form
\eqn{
v_k^e&=-C_ke^{-2N_{k+1}^2t}\mathbb P\Delta\psi_k-(\phi_k-\delta)*C_k\sum_ja_{j,k}(x)\Psi_{j,k}(x,t)\theta_j\\
&\qquad+C_kN_k^{-2}\phi_k*\sum_j\sum_{\substack{|\alpha|+|\beta|=2\\|\beta|<2}}c_{\alpha,\beta}\d^\alpha (a_{j,k}\varphi_j(M_kx))\d^\beta(\sin(N_k\eta_j\cdot x))\theta_je^{-N_k^2t}
}
for $k\in\{1,2,\ldots,k_*-1\}$,
\eqn{
v_{k_*}^e&=C_{k_*}N_{k_*}^{-2}\sum_j\sum_{\substack{|\alpha|+|\beta|=2\\|\beta|<2}}c_{\alpha,\beta}\d^\alpha (a_{j,{k_*}}\varphi_j(M_{k_*}x))\d^\beta(\sin(N_{k_*}\eta_j\cdot x))\theta_je^{-N_{k_*}^2t}
}
at $k=k_*$, and vanishes identically at $k=0$,
where $c_{\alpha,\beta}$ are fixed coefficients.

Using \eqref{C_formula}, \eqref{psi-bound-final}, \eqref{final-a-bounds}, and \eqref{mikado-pointwise-bounds}, we estimate
\eqn{
\|\grad^nv_k^e\|_\infty&\lesssim_{\Theta_*} N_k^{1+n}e^{-2N_{k+1}^2t}+(N_{k-1}+N_k^2\ell_k)N_k^{n}e^{-N_k^2t}.
}
In the case $k=k_*$, we also record the sharper estimate
\eq{\label{vk*e-bound}
\|v_{k_*}^e\|_\infty&\lesssim C_{k_*}N_{k_*}^{-1}M_{k_*}e^{-N_{k_*}^2t}.
}
It follows that
\eq{\label{error-interaction-bounds}
\|\grad^n(v_k^e\otimes v_k^e)\|_\infty+\|\grad^n(v_k^e\odot v_k^p)\|_\infty&\lesssim_{\Theta_*} N_k^{2+n}e^{-2N_{k+1}^2t}+N_k^{\frac94+n}N_{k+1}^{-\frac14}e^{-N_k^2t}.
}
Next, by \eqref{vk_bounds},
\eq{\label{mixed-interaction-bound}
\sum_{k_1\neq k_2}\|\grad^n(v_{k_1}\otimes v_{k_2})\|_{\infty}&\lesssim_{\Theta_*} \sum_{k'<k}N_{k'}N_{k}^{1+n}e^{-N_{k}^2t}\lesssim N_{k-1}N_k^{1+n}e^{-N_k^2t}.
}
Finally, observe that $v_0^p$ is a shear flow at each $t$ so $\div v_0^p\otimes v_0^p$ identically vanishes. Returning to \eqref{full-f}, we use \eqref{E1_est}, \eqref{E2_est}, \eqref{Rkhigh-final-estimate}, \eqref{error-interaction-bounds}, and \eqref{mixed-interaction-bound} to conclude
\eq{
\|\grad^n f\|_{C^\alpha}&\lesssim_{\Theta_*}\sum_{k=0}^{k_*}\Big(N_k^{n+\alpha}(N_kM_k\mathbf1_{k\geq1}+N_k^{\frac94}N_{k+1}^{-\frac14}+N_{k-1}N_k)e^{-N_k^2t}\nonumber\\
&\qquad+(\ell_kN_k^{1+n+\alpha}N_{k+1}^2\mathbf1_{k\geq1}+N_k^{2+n+\alpha}+N_{k+1}^{2+n+\alpha}N_kM_{k+1}^{-1})e^{-N_{k+1}^2t}\Big)\nonumber\\
&\lesssim \sum_{k=0}^{k_*}N_k^{\frac94+n+\alpha}N_{k+1}^{-\frac14}e^{-N_k^2t}+\sum_{k=1}^{k_*}\ell_kN_kN_{k+1}^{2+n+\alpha}e^{-N_{k+1}^2t}\nonumber\\
&\qquad+N_1^{2+n+\alpha}N_0M_1^{-1}e^{-N_1^2t},\label{grad-n-f-bound-preliminary-sum}
}
using the definitions of $N_k$ and $\ell_k$ to remove the dominated terms. Note that the extra term in the final line appears because the dominant term with decay rate $e^{-N_{k+1}^2t}$ is excluded when $k=0$; instead, in that case, the term with amplitude $N_{k+1}^{2+n+\alpha}N_kM_{k+1}^{-1}$ dominates. Note also that we have moved some $n+\alpha$ powers from $N_k$ onto $N_{k+1}$ which clearly can only increase the upper bound.

Next we observe that from the definitions of the frequency scales, for $\beta\in\mathbb R$ to be specified,
\eqn{
N_k^{\frac94+n+\alpha}N_{k+1}^{-\frac14}&= N_0^{\alpha+2\beta}N_k^{2+n-2\beta}\left(\frac{N_k}{N_0}\right)^{\alpha+2\beta}\left(\frac{N_k}{N_{k+1}}\right)^{1/4}\\
&\sim N_0^{\alpha+2\beta}N_k^{2+n-2\beta}A^{(b^k-1)(\alpha+2\beta-\frac{b-1}4)-\frac{b-1}4}.
}
Recall that $b>1$, so taking $\alpha,\beta>0$ small enough that $\alpha+2\beta<\frac{b-1}4$, we see that the power on $A$ is decreasing in $k$. Thus
\eqn{
N_k^{\frac94+n+\alpha}N_{k+1}^{-\frac14}&\lesssim N_0^{\alpha+2\beta}N_k^{2+n-2\beta}A^{-\frac{b-1}4}
}
for all $k\geq1$. Arguing in the same manner, we have
\eqn{
\ell_kN_kN_{k+1}^{2+n+\alpha}&\lesssim N_0^{\alpha+2\beta}N_{k+1}^{2+n-2\beta}A^{-\frac{b-1}4}
}
for all $k\geq1$, and
\eqn{
N_1^{2+n+\alpha}N_0M_1^{-1}e^{-N_1^2t}&\lesssim N_0^{\alpha+2\beta}N_1^{2+n-2\beta}A^{-(\gamma b-1)}.
}
Returning to \eqref{grad-n-f-bound-preliminary-sum} and taking $\beta>0$ smaller as needed (depending on $b$ and $\gamma$) we arrive at
\eqn{
\|\grad^nf\|_{C^\alpha}&\lesssim_{\Theta_*} A^{-\beta}N_0^{\alpha+2\beta}\sum_{k=0}^{k_*}N_k^{2+n-2\beta}e^{-N_k^2t}.
}
We fix $t>0$ and consider two regimes\footnote{Strictly speaking, there is a third regime when $t<N_{k_*}^{-2}$, but we do not require sharper estimates for such times.}: first, if $t\leq N_0^{-2}$, then the terms in the sum grow super-exponentially in $k$ while $N_0\leq N_k\leq t^{-\frac12}$, then decay even faster when $t^{-\frac12}<N_k\leq N_{k_*}$. Clearly the sum is bounded by its largest term. Because $N\mapsto N^{2+n-2\beta}e^{-N^2t}$ is maximized at $N=C(n,\beta)t^{-\frac12}$, the upper bound is $O(A^{-\beta}N_0^{\alpha+2\beta}t^{-1-\frac n2+\beta}$). Suppose instead $t>N_0^{-2}$. Then the summands strictly decay so we have the upper bound $O(A^{-\beta}N_0^{2+n+\alpha})e^{-N_0^2t}$. Combining these estimates, we conclude \eqref{final-f-bound}.
\end{proof}

Having obtained satisfactory estimates on the drift (Lemma~\ref{v-estimates-lemma}) and the force (Proposition~\ref{f-prop}), we are in a position to show the smallness of $w$ and conclude that the solution is global.

\begin{proposition}\label{w-prop}
    Let $v$, $w$, and $T_*$ be as in Definitions~\ref{v-psi-etc-definition} and \ref{full-definition}. For any $n_{max}'''\in\mathbb N$, $\epsilon>0$, and all sufficiently large $A=A(\epsilon,n_{max}''')$, we have $T_*=\infty$ and
    \eq{\label{final-w-bound}
    t^{\frac12(1+n)}\|\grad^nw(t)\|_{L^\infty}&\leq \epsilon (N_0^2t)^{\beta}\qquad\forall t\in(0,KN_0^{-2}]
    }
    for $n=0,1,\ldots,n_{max}'''$, with $\beta\in(0,\frac12)$ as in Proposition~\ref{f-prop} and a large $K>1$ to be specified.
\end{proposition}

\begin{proof}
    By \eqref{v_equation}, $w$ obeys
    \eqn{
    \d_tw-\Delta w+\mathbb P\div (w\otimes w+2v\odot w)=-\mathbb P\div f\\
    w|_{t=0}=0
    }
    so we have the Duhamel formula
    \eq{\label{w-duhamel}
    w(t)=-\int_0^te^{(t-t')\Delta}\mathbb P\div (w\otimes w+2v\odot w+f)(t')dt'
    }
    leading to, by standard multiplier estimates,
    \eqn{
    \|w(t)\|_\infty&\lesssim \int_0^t\Big((t-t')^{-\frac12}(\|w(t')\|_\infty^2+\|v(t')\|_\infty\|w(t')\|_\infty)\\
    &\qquad+(t-t')^{-\frac{1-\alpha}2}\|f(t')\|_{C^\alpha}\Big)dt'.
    }
    For simplicity, rather than continuing to work globally in time, we fix a large $K>1$ and consider $t\in[0,KN_0^{-2}]$. (The choice of $K$ will be specified at the end of the proof depending only on $\eta_*$.) For such $t$, we apply \eqref{final-f-bound} to find
    \eqn{
    t^\frac12\|w(t)\|_\infty&\lesssim_{\Theta_*} \int_0^tt^\frac12(t')^{-\frac12}(t-t')^{-\frac12}(\|v(t')\|_\infty+\|w(t')\|_\infty)(t')^\frac12\|w(t')\|_\infty dt'\\
    &\qquad+A^{-\beta}(N_0^2t)^{\frac\alpha2+\beta}.
    }
    Consider the weight $t^\frac12(t')^{-\frac12}(t-t')^{-\frac12}$. In the intervals $t'\in[0,t/2]$ and $[t/2,t]$, we use the upper bounds $(t')^{-\frac12}$ and $(t-t')^{-\frac12}$ respectively (up to a constant multiple). Enlarging the intervals back to $[0,t]$ and introducing the function $g(t)=t^\frac12\|w(t)\|_\infty$, we have
    \eqn{
    g(t)\lesssim_{\Theta_*}A^{-\beta}(N_0^2t)^{\frac\alpha2+\beta}+\int_0^t((t')^{-\frac12}+(t-t')^{-\frac12})(\|v(t')\|_\infty+\|w(t')\|_\infty)g(t')dt'.
    }
    From here we can apply the Gr\"onwall inequality~\cite[Lemma~B.3]{coic-palasek} to obtain
    \begin{equation}\begin{aligned}\label{w-gronwall-result}
    &t^\frac12\|w(t)\|_\infty\lesssim_{\Theta_*} A^{-\beta}(N_0^2t)^{\frac\alpha2+\beta}\exp\Big(O_{\Theta_*}(\|s^{-\frac12}v\|_{L_s^1L_x^\infty}+\|s^{-\frac12}w\|_{L_s^1L_x^\infty}\\
    &\qquad+(\|s^\frac12v\|_{L_s^\infty L_x^\infty}+\|s^\frac12w\|_{L_s^\infty L_x^\infty})(\|w\|_{L_s^2L_x^\infty}^2+\|v\|_{L_s^2L_x^\infty}^2))\Big)
    \end{aligned}\end{equation}
    for every $t\in(0,\min(KN_0^{-2},T_*))$ where, for brevity, we have written $L_s^pL_x^\infty$ in place of $L^p([0,t];L^\infty(\Td))$. Heuristically, as long as $w$ is bounded in a scale-invariant sense up to time $t$, it can be made arbitrarily small by the choice of $A$, and the estimates can be propagated forward in time. We make this reasoning precise using a bootstrap argument.
    
    For a fixed $T_1\in(0,T_*)$, consider the hypothesis
    \eq{\label{bootstrap_hypothesis}
    \|w(t)\|_{\infty}\leq \epsilon N_0^{2\beta}t^{\beta-\frac12}\qquad\forall t\in(0,T_1].
    }
    Recall that $u$ is a local solution of \eqref{nse} with smooth data, so $\|u\|_{L_{t,x}^\infty([0,T_*/2]\times\Td)}<\infty$. From Definition~\ref{v-psi-etc-definition} and \eqref{vk_bounds}, the same is true of $v$; thus \eqref{bootstrap_hypothesis} is true for all $T_1$ sufficiently small (recalling $\beta<\frac12$). Furthermore, it is easy to see that the map $t\mapsto t^{\frac12-\beta}\|w(t)\|_{\infty}$ is continuous on $(0,T_*)$. Thus, to bootstrap \eqref{bootstrap_hypothesis} to the full interval $(0,\min(KN_0^{-2},T_*))$, it suffices to show that for each $T_1\in(0,\min(KN_0^{-2},T_*))$, the hypothesis \eqref{bootstrap_hypothesis} implies
    \eq{\label{stronger-hypothesis-bootstrap-claim}
   \|w(t)\|_{\infty}\leq \frac\epsilon2N_0^{2\beta} t^{\beta-\frac12}\qquad\forall t\in(0,T_1].
    }
    Indeed, \eqref{bootstrap_hypothesis} easily gives
    \eq{\label{bootstrap-consequences}
    \|w\|_{L^2([0,t];L^\infty)}\leq(2\beta)^{-\frac12} (N_0^2T_1)^\beta\epsilon,\qquad\|s^{-\frac12}w\|_{L^1([0,t];L^\infty)}\leq\beta^{-1}(N_0^2T_1)^\beta\epsilon
    }
    for all $t\in(0,T_1]$. Inserting \eqref{bootstrap-consequences} and \eqref{critical_v_bounds} into \eqref{w-gronwall-result}, we have
    \eqn{
    \|w(t)\|_{\infty}\lesssim_{\Theta_*,K}A^{-\beta}t^{-\frac12}(N_0^2t)^\beta \exp(O_{\Theta_*,k_*,K}(1))\qquad\forall t\in(0,T_1]
    }
    from which the claim \eqref{stronger-hypothesis-bootstrap-claim} follows upon taking $A$ sufficiently large depending on $\Theta_*$, $k_*$, $K$, and $\epsilon$.
    
    This completes the bootstrap argument; we conclude that \eqref{bootstrap_hypothesis} holds with $T_1=\min(KN_0^{-2},T_*)$. Because $v\in L_{t,x}^\infty(\Td\times[0,\infty))$, it follows from standard local theory that $T_*\geq KN_0^{-2}$. Then using \eqref{vk_bounds}, \eqref{bootstrap_hypothesis}, and H\"older's inequality, we compute
    \eqn{
    \|u(KN_0^{-2}/2)\|_{3}&\leq \|v(KN_0^{-2}/2)\|_{3}+\|w(KN_0^{-2}/2)\|_{3}\\
    &\lesssim \sum_{k=0}^{k_*}N_k\exp(-K(N_k/N_0)^2/2)+\epsilon K^{\beta-\frac12}N_0\\
    &\lesssim N_0(\exp(-K/2)+\epsilon K^{\beta-\frac12}),
    }
    using in the last inequality that the sum is dominated by its first term due to the rapid decay. Choosing $K$ large depending\footnote{A more dimensionally consistent approach in the case where $N_0$ is much larger than the minimum frequency on $\Td$ would involve localizing $v_0(x,t)$ to a domain of volume $\sim N_0^{-3}$, but this would unnecessarily complicate the argument.} on $N_0$, we can arrange that the $L^3$ norm is smaller than any absolute constant. Then Kato's mild solution theory applied at $t=KN_0^{-2}/2$ guarantees that $u$ is global, as claimed.

    The higher derivative estimates can be obtained by a bootstrap argument from the $n=0$ case; given the claim for some $n\geq0$, one can take $n$ derivatives of \eqref{w-duhamel} and apply the same Gr\"onwall argument to the quantity $t^{(n+1)/2}\|w(t)\|_{\infty}$. The ideas are standard and we omit the details.
\end{proof}

\subsection{Final estimates}\label{final-estimates-conclusion-subsection}

Now we conclude the proof.

\begin{proof}[Proof of Theorem~\ref{second-theorem}]
    
Let $v_k$, $\psi_k$, etc.\ be as in Definition~\ref{v-psi-etc-definition}. Let $\psi(\xi/N)$ be the symbol for the Fourier multiplier $P_N$. Toward estimating the norm of the data, we have
\eq{
P_N(a_{j,k_*}\Psi_{j,k_*})(x,0)&=\sum_{\xi\in\mathbb Z^3}\psi(\xi/N)\mathcal F_{y\to\xi}\Big(a_{j,k_*}\varphi_j(M_{k_*}y)\sin(N_{k_*}\eta_j\cdot y)\Big)(\xi)e^{ix\cdot\xi}\nonumber\\
&=\im\sum_{\xi\in\mathbb Z^3}\psi(\xi/N)\mathcal F\Big(a_{j,k_*}\varphi_j(M_{k_*}\cdot)\Big)(\xi-N_{k_*}\eta_j)e^{ix\cdot\xi}\label{lp-formula}
}
by Definition~\ref{v-psi-etc-definition} and some straightforward algebra. There are two regimes for the dyadic scale $N$. First, consider the case where we project to frequencies away from the primary support of $v_{k_*}$ (the only leakage coming from the rapidly decaying Fourier tail of $a_{j,k_*}(x)\varphi_j(M_{k_*}x)$). For $N\notin(N_{k_*}/2,2N_{k_*})$, we apply \eqref{lp-formula}, \eqref{mikado-pointwise-bounds}, and \eqref{final-a-bounds} to find
\eq{
\|P_N\mathbb P\Delta\psi_{k_*}(0)\|_\infty&\lesssim (N/N_{k_*})^2\sup_j\|P_N(a_{j,k_*}\Psi_{j,k_*})(0)\|_{\infty}\nonumber\\
&\leq (N/N_{k_*})^2\sum_{\frac 35N\leq|\xi|\leq\frac{19}{10}N}|\xi-N_{k_*}\eta_j|^{-4}\|\Delta^2(a_{j,k_*}\varphi_j(M_{k_*}x))\|_{\infty}\nonumber\\
&\lesssim (N/N_{k_*})^2\sum_{\frac 35N\leq|\xi|\leq\frac{19}{10}N}\max\{N_{k_*},|\xi|\}^{-4}M_{k_*}^4\nonumber\\
&\lesssim (M_{k_*}/N_{k_*})^4\min\left(\frac{N_{k_*}}{N},\frac{N}{N_{k_*}}\right)^{2}.\label{error-frequency-high-low-bound}
}
These error terms are small by the observation \eqref{M-over-N-fact-small}, and summable over the dyadics due to the $\min()$ factor.

What remains is the primary contribution to the norm where we project to frequency shells near $N_{k_*}$. Recall the decomposition $v_k=v_k^p+v_k^e$ from the proof of Proposition~\ref{f-prop}. For $N\in(N_{k_*}/2,2N_{k_*})$, we use \eqref{mikado-pointwise-bounds}, \eqref{vk*e-bound}, \eqref{a-precise-upper-lower-bounds-final}, \eqref{M-over-N-fact-small} to find
\eqn{
\|P_Nv_{k_*}(0)\|_\infty&\leq\|v_{k_*}^p(0)\|_\infty+\|v_{k_*}^e(0)\|_\infty\\
&\leq C_{k_*}\big(\sup_j\|a_{j,k_*}\|_\infty+O\Big(\frac{M_{k_*}}{N_{k_*}}\Big)\big)\\
&\leq 32000C_{k_*}.
}
Choosing $k_*$ sufficiently large depending on $\Theta_*$, we can arrange that
\eqn{
C_{k_*}=N_{k_*}\left(\frac{|\Theta_*|}{|\eta_*|}\right)^{2^{-k_*}}\leq \frac{101}{100}N_{k_*}
}
by \eqref{C_formula}. Thus
\eq{\label{mid-frequency-bound}
\|P_Nv_{k_*}(0)\|_\infty&\leq33000N_{k_*}\qquad \forall N\in(N_{k_*}/2,2N_{k_*}).
}

From Definitions~\ref{v-psi-etc-definition} and \ref{full-definition}, we have
\eqn{
u^0(x)&=v(x,0)=v_{k_*}(x,0)=C_{k_*}\mathbb P\Delta\psi_{k_*}(x,0).
}
Finally we can complete the estimate, recalling the definition of the norm from \S\ref{definitions-notation-section}. Using \eqref{error-frequency-high-low-bound} and \eqref{mid-frequency-bound}, we have
\eqn{
\|u^0\|_{B^{-1}_{\infty,1}}&\leq \sum_{N\in(N_{k_*}/2,2N_{k_*})}N^{-1}\|P_Nv_{k_*}(x,0)\|_\infty\\
&\qquad+\sum_{N\not\in(N_{k_*}/2,2N_{k_*})}N^{-1}C_{k_*}\|P_N\mathbb P\Delta\psi_{k_*}(x,0)\|_\infty\\
&\leq 33000N_{k_*}\left(\frac2{N_{k_*}}+\frac1{N_{k_*}}\right)+O\Big(\left(\frac{M_{k_*}}{N_{k_*}}\right)^4\sum_N\min\Big(\left(\frac{N_{k_*}}{N}\right)^3,\frac{N}{N_{k_*}}\Big)\Big)\\
&=99000+O(A^{-(1-\gamma)b^{k_*}}).
}
We emphasize that the sums are taken over dyadic integers in the indicated ranges. Clearly the second term can be made small by the choice of $A$, and we reach the claimed upper bound. 

Now we establish that $u$ becomes large as claimed. Following the definitions in Definition~\ref{v-psi-etc-definition}, \eqref{C_formula}, and Lemma~\ref{construct_coefficients_lemma},
\eqn{
v_0(x,t)&=C_0(1-e^{-2N_1^2t})\mathbb P\Delta\psi_0(x,t)\\
&=\frac{C_0a_{1,0}}{N_0^2}(1-e^{-2N_1^2t})\mathbb P\Delta\left(\sin(x\cdot\eta_*)\frac{\Theta_*}{|\Theta_*|}\right)e^{-|\eta_*|^2t}\\
&=(1-e^{-2N_1^2t})\Theta_*\sin(x\cdot\eta_*)e^{-|\eta_*|^2t}.
}
Defining $E$ to satisfy the statement of Theorem~\ref{second-theorem}, we have
\eqn{
E(t,x)&=-\Theta_*\sin(x\cdot\eta_*)e^{-(2N_1^2+|\eta_*|^2)t}+\sum_{k=1}^{k_*} v_k(x,t)+w(x,t).
}
By \eqref{vk_bounds} and \eqref{final-w-bound},
\eqn{
\|\grad^nE(t)\|_\infty&\lesssim_{\Theta_*,\eta_*,n}e^{-2N_1^2t}+\sum_{k=1}^{k_*} N_k^{1+n}e^{-N_k^2t}+\epsilon(N_0^2t)^\beta t^{-\frac12(1+n)}\\
&\lesssim_{\eta_*,n} e^{-(N_1/|\eta_*|)^2}+\sum_{k=1}^{k_*}N_k^{1+n}e^{-(N_k/|\eta_*|)^2/2}+\epsilon\\
&\lesssim_{\eta_*,n} N_1^{-10}+\epsilon
}
for $t\in[|\eta_*|^{-2},2|\eta_*|^{-2}]$, recalling that $N_0=\big\lceil|\eta_*|\big\rceil$, and that by choosing $A$ sufficiently large, we can arrange that $N_k\gg|\eta_*|$ for all $k\geq1$. Taking $A$ larger as necessary and $\epsilon>0$ smaller (as allowed by Proposition~\ref{w-prop}), we arrive at the claimed bound on $\grad^nE$.
\end{proof}

\section{Proof of Corollaries~\ref{a-priori-estimate-corollary} and \ref{kt-vmo-bmo-corollary}}\label{corollary-section}

Having completed the proof of Theorem~\ref{second-theorem}, we can easily prove the corollaries.

\begin{proof}[Proof of Corollary~\ref{a-priori-estimate-corollary}]

Toward a disproof of \eqref{theorem-apriori-bound-e1}, we apply Theorem~\ref{second-theorem} with $\Theta_*=ne_1$, $\eta_*=me_2$, $\epsilon_*=1$, and $n_{max}=0$ for each $n,m\in\mathbb N$. Thus we arrive at a sequence of smooth global solutions $u_{n,m}$ with $\|u_{n,m}^0\|_{B^{-1}_{\infty,1}}=O(1)$ and
\eq{\label{error}
\|u_{n,m}(x,m^{-2})-ne_1\sin(mx_2)e^{-1}\|_\infty\leq1.
}
Thanks to $\|u_{n,m}^0\|_{B^{-1}_{\infty,1}}\leq10^5$, we have $\|u_{n,m}^0\|_{BMO^{-1}}\leq r_0$ with $r_0$ an absolute constant coming from the embedding \eqref{besov-bmo-inclusion}.
Further,
\eqn{
\|u_{n,m}(m^{-2})\|_{BMO^{-1}}&\geq \frac{n}{e}\|\sin(mx_2)\|_{BMO^{-1}}\\
&\qquad-\|u_{n,m}(x,m^{-2})-ne_1\sin(mx_2)e^{-1}\|_{BMO^{-1}}.
}
The $BMO^{-1}$ norm is controlled by $B^{-1}_{\infty,2}$ which, in the context of zero-average functions on $\Td$, is clearly controlled by $L^\infty$. Note also that $\|\sin(mx_2)\|_{BMO^{-1}}\sim |m|^{-1}$ due to \eqref{besov-bmo-inclusion}. By \eqref{error}, we conclude
\eqn{
\|u_{n,m}(m^{-2})\|_{BMO^{-1}}\gtrsim \frac nm-1.
}
We contradict \eqref{theorem-apriori-bound-e1} for a particular $T$ and $f$ by first taking $m$ large depending on $T$, then $n$ large depending on $m$ and $f(r_0)$.

The proofs of the impossibility of \eqref{theorem-a-priori-bound-e2}--\eqref{theorem-apriori-bound-prodiserrin} are nearly identical so we omit the details.
\end{proof}

\begin{proof}[Proof of Corollary~\ref{kt-vmo-bmo-corollary}]
    Consider the same sequence of global solutions $u_{n,m}$ constructed in Corollary~\ref{a-priori-estimate-corollary}. Recall that the initial data belong to a ball $B_{BMO^{-1}}(0,r_0)$ with $r_0>0$ an absolute constant. Toward contradiction, suppose that for some $T,R>0$, the Picard method yields mild solutions in $B_{X_T}(0,R)$ from data $u_{n,m}^0$ for every $n,m\in\mathbb N$. By the local theory, those solutions agree with $u_{n,m}$ on $[0,T]$ (recalling the data is smooth). But, by the arguments in the proof of Corollary~\ref{a-priori-estimate-corollary}, we have
    \eqn{
    \|u_{n,m}(m^{-2})\|_\infty\gtrsim n-1
    }
   and therefore
    \eqn{
    \sup_{t\in[0,T]}t^\frac12\|u_{n,m}(t)\|_\infty \gtrsim \frac {n-1}m,
    }
    upon taking $m$ large depending on $T$. Note that this time-weighted supremum norm is one term in the definition of $X_T$. Thus, taking $n$ large depending on $m$ and $R$, we conclude $u_{n,m}\not\in B_{X_T}(0,R)$, a contradiction.
\end{proof}

\appendix

\section{Rank-one decomposition and Mikado flows}\label{mikado-appendix}

Because we are interested in computing an explicit and reasonable value for the size of the initial data, it becomes necessary to specify the objects in Lemma~\ref{nash-lemma} and Definition~\ref{mikado-definition} concretely.

\begin{proof}[Proof of Lemma~\ref{nash-lemma}]
    One can take $\theta_1=(\frac45,0,\frac35)$, $\theta_2=(0,-\frac35,\frac45)$, $\theta_3=(0,\frac45,\frac35)$, $\theta_4=(\frac35,0,-\frac45)$, $\theta_5=(\frac35,\frac45,0)$, and $\theta_6=(-\frac45,\frac35,0)$. Then \eqref{nash-lemma-span-equation} is reduced to the claim that $(\theta_j\otimes\theta_j)_{j=1}^6$ are linearly independent and generate $\Id$ with strictly positive coefficients. Indeed, the positivity must persist for all symmetric matrices $M$ in a neighborhood of $\Id$; thus we can define $\Gamma_j$ such that $\Gamma_j^2(M)$ are these coefficients.
    
    It is easy to see that $(\theta_j\otimes\theta_j)_j$ sum to $2\Id$, so indeed there is a neighborhood where the coefficients are close to $\frac12$ and therefore positive. To compute the size of the neighborhood, one can solve the six-dimensional linear system explicitly to reduce \eqref{nash-lemma-span-equation} to
    \eqn{
    \Gamma_j(\Id+\epsilon)^2=\frac12+\sum_{1\leq k\leq l\leq3}b_{jkl}\epsilon_{kl}
    }
    for some constants $b_{jkl}$ and all sufficiently small symmetric matrices $\epsilon$. We compute, for instance in the case $j=1$, that $b_{111}=\frac12$, $b_{112}=\frac7{24}$, $b_{113}=\frac{25}{24}$, $b_{122}=\frac12$, $b_{123}=\frac7{24}$, and $b_{133}=\frac12$. More generally, for every $j$, $(b_{jkl})_{kl}$ is a permutation of these same six values, up to signs. In particular, for every $j=1,2\ldots,6$, we have $\sum_{k,l}|b_{jkl}|=25/8$. Thus, when $\|\epsilon\|<1/7$, we have $|\Gamma_j^2-\frac12|\leq\frac{25}{56}$ (recalling that we use the entry-wise maximum norm on symmetric matrices). It follows that each $\Gamma_j^2$ avoids $0$ so the square root can be taken smoothly, and we conclude \eqref{gamma-bounds}.
\end{proof}

\begin{proof}[Proof of Lemma~\ref{mikado-delta-lemma}]

We define the positions
\begin{align*}
x_1&=\left(\frac{21}{100},\frac{26}{25},\frac{47}{50}\right),\,x_2=\left(\frac{37}{50},\frac{467}{100},\frac{357}{100}\right),\,x_3=\left(\frac{7}{5},\frac{126}{25},\frac{91}{100}\right),\\
x_4&=\left(\frac{393}{100},\frac{104}{25},\frac{341}{100}\right),\,x_5=\left(\frac{3}{4},\frac{617}{100},\frac{153}{25}\right),\,x_6=\left(\frac{261}{100},\frac{307}{50},\frac{339}{100}\right)\in\Td.
\end{align*}
For $j_1\neq j_2$, the distance between the periodic lines $\mathcal L_{j_1}=x_{j_1}+\mathbb R\theta_{j_1}$ and $\mathcal L_{j_2}=x_{j_2}+\mathbb R\theta_{j_2}\bmod2\pi\mathbb Z^3$ is given by
\eqn{
\dist_{\Td}(\mathcal L_{j_1},\mathcal L_{j_2})=\min_{m\in\mathbb Z^3}\left|(x_{j_1}-x_{j_2}+2\pi m)\cdot\frac{\theta_{j_1}\times\theta_{j_2}}{|\theta_{j_1}\times\theta_{j_2}|}\right|.
}
Given that $5\theta_j\in\mathbb Z^3$, in practice it suffices to minimize over only $m\in\{0,1,2,3,4\}^3$. Computing them explicitly, we find
\begin{align*}
 \dist_{\Td}(\mathcal L_{j_1},\mathcal L_{j_2})= 
\begin{pmatrix}[1.7]
 0 & \frac{2800 \pi -8487}{100 \sqrt{481}} & \frac{1569-400 \pi }{100 \sqrt{34}} & \frac{78}{25} & \frac{4000 \pi -12257}{100 \sqrt{481}} & \frac{30 \pi -89}{5 \sqrt{41}} \\
 * & 0 & \frac{33}{50} & \frac{2 (75 \pi -128)}{25 \sqrt{41}} & \frac{4079-1200 \pi }{100 \sqrt{481}} & \frac{3 (40 \pi -73)}{20 \sqrt{34}} \\
 * & * & 0 & \frac{8 (49-15 \pi )}{5 \sqrt{481}} & \frac{297-80 \pi }{20 \sqrt{41}} & \frac{1559-400 \pi }{100 \sqrt{481}} \\
 * & * & * & 0 & \frac{200 \pi -531}{50 \sqrt{34}} & \frac{783}{50 \sqrt{481}} \\
 * & * & * & * & 0 & \frac{273}{100} \\
 * & * & * & * & * & 0 \\
\end{pmatrix}
_{j_1,j_2}
\end{align*}
where we omit the lower triangle, which is symmetric. One finds that the smallest off-diagonal entry is $\frac{8 (49-15 \pi )}{5 \sqrt{481}}\approx 0.1368$.
    
\end{proof}

\bibliographystyle{abbrv}
\bibliography{references}

\end{document}